\documentclass[10pt,oneside]{amsart}
\usepackage{amssymb, amsmath, amsthm, amsfonts}
\usepackage{stmaryrd}
\usepackage{mathrsfs,comment}
\usepackage{graphicx}
\usepackage{placeins}
\usepackage{rotating}
\usepackage{tikz}
\usepackage{float}
\usepackage{hvfloat}
\usepackage{caption}
\usepackage{pdflscape}
\usepackage{notoccite}
\usepackage{hyperref}  
\usepackage{url}
\usepackage[all,arc,2cell]{xy}
\UseAllTwocells
\usepackage{enumerate}
\usepackage{chngcntr}
 \usepackage{lineno}
 \usepackage{blindtext}
\usepackage{verbatim}
\usepackage{soul}
\usepackage[normalem]{ulem}
\usepackage{lscape}
\usepackage{geometry}
\usepackage{pifont}

\usepackage{tikz}
\usepackage{tikz-cd} 
\usepackage[makeroom]{cancel}

\hypersetup{%
  bookmarksnumbered=true,%
  bookmarks=true,%
  colorlinks=true,%
  linkcolor=blue,%
  citecolor=blue,%
  filecolor=blue,%
  menucolor=blue,%
  pagecolor=blue,%
  urlcolor=blue,%
  pdfnewwindow=true,%
  pdfstartview=FitBH}

\def\@url#1{{\tt\def~{\lower3.5pt\hbox{\char'176}}\def\_{\char'137}#1}}

\def\makeautorefname#1#2{\expandafter\def\csname#1autorefname\endcsname{#2}}
\makeautorefname{equation}{Equation}%
\makeautorefname{footnote}{footnote}%
\makeautorefname{item}{item}%
\makeautorefname{figure}{Figure}%
\makeautorefname{table}{Table}%
\makeautorefname{part}{Part}%
\makeautorefname{appendix}{Appendix}%
\makeautorefname{chapter}{Chapter}%
\makeautorefname{section}{Section}%
\makeautorefname{subsection}{Section}%
\makeautorefname{subsubsection}{Section}%
\makeautorefname{paragraph}{Paragraph}%
\makeautorefname{subparagraph}{Paragraph}%
\makeautorefname{theorem}{Theorem}%
\makeautorefname{theo}{Theorem}%
\makeautorefname{thm}{Theorem}%
\makeautorefname{addendum}{Addendum}%
\makeautorefname{addend}{Addendum}%
\makeautorefname{add}{Addendum}%
\makeautorefname{maintheorem}{Main theorem}%
\makeautorefname{mainthm}{Main theorem}%
\makeautorefname{corollary}{Corollary}%
\makeautorefname{claim}{Claim}%
\makeautorefname{corol}{Corollary}%
\makeautorefname{coro}{Corollary}%
\makeautorefname{cor}{Corollary}%
\makeautorefname{lemma}{Lemma}%
\makeautorefname{lemm}{Lemma}%
\makeautorefname{lem}{Lemma}%
\makeautorefname{sublemma}{Sublemma}%
\makeautorefname{sublem}{Sublemma}%
\makeautorefname{subl}{Sublemma}%
\makeautorefname{proposition}{Proposition}%
\makeautorefname{proposit}{Proposition}%
\makeautorefname{propos}{Proposition}%
\makeautorefname{propo}{Proposition}%
\makeautorefname{prop}{Proposition}%
\makeautorefname{property}{Property}
\makeautorefname{proper}{Property}
\makeautorefname{scholium}{Scholium}%
\makeautorefname{step}{Step}%
\makeautorefname{conjecture}{Conjecture}%
\makeautorefname{conject}{Conjecture}%
\makeautorefname{conj}{Conjecture}%
\makeautorefname{question}{Question}
\makeautorefname{questn}{Question}
\makeautorefname{quest}{Question}
\makeautorefname{ques}{Question}
\makeautorefname{qn}{Question}
\makeautorefname{definition}{Definition}%
\makeautorefname{defin}{Definition}%
\makeautorefname{defi}{Definition}%
\makeautorefname{def}{Definition}%
\makeautorefname{defn}{Definition}%
\makeautorefname{dfn}{Definition}%
\makeautorefname{notation}{Notation}
\makeautorefname{nota}{Notation}
\makeautorefname{notn}{Notation}
\makeautorefname{remark}{Remark}%
\makeautorefname{rema}{Remark}%
\makeautorefname{rem}{Remark}%
\makeautorefname{rmk}{Remark}%
\makeautorefname{rk}{Remark}%
\makeautorefname{remarks}{Remarks}%
\makeautorefname{rems}{Remarks}%
\makeautorefname{rmks}{Remarks}%
\makeautorefname{rks}{Remarks}%
\makeautorefname{example}{Example}%
\makeautorefname{examp}{Example}%
\makeautorefname{exmp}{Example}%
\makeautorefname{exmps}{Examples}%
\makeautorefname{exam}{Example}%
\makeautorefname{exa}{Example}%
\makeautorefname{algorithm}{Algorith}%
\makeautorefname{algo}{Algorith}%
\makeautorefname{alg}{Algorith}%
\makeautorefname{axiom}{Axiom}%
\makeautorefname{axi}{Axiom}%
\makeautorefname{ax}{Axiom}%
\makeautorefname{case}{Case}%
\makeautorefname{claim}{Claim}%
\makeautorefname{clm}{Claim}%
\makeautorefname{assumption}{Assumption}%
\makeautorefname{assumpt}{Assumption}%
\makeautorefname{conclusion}{Conclusion}%
\makeautorefname{concl}{Conclusion}%
\makeautorefname{conc}{Conclusion}%
\makeautorefname{condition}{Condition}%
\makeautorefname{condit}{Condition}%
\makeautorefname{cond}{Condition}%
\makeautorefname{construction}{Construction}%
\makeautorefname{construct}{Construction}%
\makeautorefname{const}{Construction}%
\makeautorefname{cons}{Construction}%
\makeautorefname{criterion}{Criterion}%
\makeautorefname{criter}{Criterion}%
\makeautorefname{crit}{Criterion}%
\makeautorefname{exercise}{Exercise}%
\makeautorefname{exer}{Exercise}%
\makeautorefname{exe}{Exercise}%
\makeautorefname{problem}{Problem}%
\makeautorefname{problm}{Problem}%
\makeautorefname{probm}{Problem}%
\makeautorefname{prob}{Problem}%
\makeautorefname{solution}{Solution}%
\makeautorefname{soln}{Solution}%
\makeautorefname{sol}{Solution}%
\makeautorefname{summary}{Summary}%
\makeautorefname{summ}{Summary}%
\makeautorefname{sum}{Summary}%
\makeautorefname{operation}{Operation}%
\makeautorefname{oper}{Operation}%
\makeautorefname{observation}{Observation}%
\makeautorefname{observn}{Observation}%
\makeautorefname{obser}{Observation}%
\makeautorefname{obs}{Observation}%
\makeautorefname{ob}{Observation}%
\makeautorefname{convention}{Convention}%
\makeautorefname{convent}{Convention}%
\makeautorefname{conv}{Convention}%
\makeautorefname{cvn}{Convention}%
\makeautorefname{warning}{Warning}%
\makeautorefname{warn}{Warning}%
\makeautorefname{note}{Note}%
\makeautorefname{fact}{Fact}%

  \makeatletter
                   \let\c@lemma\c@theorem
                  \makeatother

\newtheorem{thm}{Theorem}[subsection]
\newtheorem{cor}{Corollary}[subsection]
\newtheorem{prop}{Proposition}[subsection]
\newtheorem{lem}{Lemma}[subsection]

\theoremstyle{definition}
\newtheorem{defn}{Definition}[subsection]
\newtheorem{exmp}{Example}[subsection]

\newtheorem{rem}{Remark}[subsection]

\makeatletter
\let\c@lem=\c@thm
\let\c@cor=\c@thm
\let\c@prop=\c@thm
\let\c@lem=\c@thm
\let\c@defn=\c@thm
\let\c@exmp=\c@thm
\let\c@exmps=\c@thm
\let\c@rem=\c@thm
\let\c@warn=\c@thm
\let\c@claim=\c@thm
\let\c@quest=\c@thm
\makeatother

\numberwithin{equation}{subsection}

\newcommand{\Z}{\mathbb{Z}}

\newcommand{\F}{\mathbb{F}}

\newcommand{\gmM}{{}^{g}\!\m{M}}
\newcommand{\gmR}{{}^{g}\!\m{R}}
\DeclareSymbolFontAlphabet{\scr}{rsfs}

\def\quickop#1{\expandafter\newcommand\csname #1\endcsname{\operatorname{#1}}}
\quickop{Hom} \quickop{End} \quickop{Aut} \quickop{Tel} \quickop{Mic} 
\quickop{Ext} \quickop{Tor} \quickop{Cotor} \quickop{Id} \quickop{Coker} \quickop{Ker}
\quickop{Lim} \quickop{Colim} \quickop{Holim} \quickop{Hocolim}
\quickop{id} \quickop{tel} \quickop{mic} \quickop{coker} 
\quickop{colim}

\DeclareMathOperator*{\holim}{holim}

\DeclareMathOperator{\THH}{THH}
\DeclareMathOperator{\HH}{HH}

\newcommand{\m}[1]{\underline{#1}}

\newcommand{\Om}{\Omega}

\renewcommand\Id{\textrm{Id}}

\definecolor{darkspringgreen}{rgb}{0.09, 0.45, 0.27}
\definecolor{darkterracotta}{rgb}{0.8, 0.31, 0.36}
	\definecolor{darkcoral}{rgb}{0.8, 0.36, 0.27}
	\definecolor{indiagreen}{rgb}{0.07, 0.53, 0.03}
	\definecolor{mountainmeadow}{rgb}{0.19, 0.73, 0.56}
	\definecolor{mountbattenpink}{rgb}{0.6, 0.48, 0.55}
	\definecolor{palatinatepurple}{rgb}{0.41, 0.16, 0.38}
	\definecolor{cinnamon}{rgb}{0.82, 0.41, 0.12}
	\definecolor{chocolate}{rgb}{0.82, 0.41, 0.12}

\usepackage{scalerel}
\newcommand{\sm}{\wedge}

\newcommand{\cA}{\mathcal{A}}

\newcommand{\abs}[1]{\lvert #1\rvert}
\newcommand{\sqr}{\square}

\newcommand{\Algcat}[1]{{#1}\textrm{-Alg}}

\newcommand{\scalestar}{{\scaleto{\star}{4pt}}}

\newcommand{\mhyphen}{\text{ -- }}
\usepackage[utf8]{inputenc}

\title{Computational tools for twisted topological Hochschild homology of equivariant spectra}

\author[Adamyk]{Katharine Adamyk}
\address[Adamyk]{Department of Mathematics, Western University, London, ON, Canada}
\author[Gerhardt]{Teena Gerhardt}
\address[Gerhardt]{Department of Mathematics, Michigan State University, East Lansing, MI 48824 }
\author[Hess]{Kathryn Hess}
\address[Hess]{SV UPHESS BMI, \'Ecole Polytechnique F\'ed\'erale de Lausanne, 1015 Lausanne, Switzerland}
\author[Klang]{Inbar Klang}
\address[Klang]{Department of Mathematics, Columbia University, New York, NY 10027}
\author[Kong]{Hana Jia Kong}
\address[Kong]{Department of Mathematics, The University of Chicago, Chicago, IL 60637}

\begin{document}

\maketitle
\markboth{ADAMYK, GERHARDT, HESS, KLANG, AND KONG}{COMPUTATIONAL TOOLS FOR TWISTED THH OF EQUIVARIANT SPECTRA}

\begin{abstract}
Twisted topological Hochschild homology of  $C_n$-equivariant spectra was introduced by Angeltveit, Blumberg, Gerhardt, Hill, Lawson, and Mandell, building on the work of Hill, Hopkins, and Ravenel on norms in equivariant homotopy theory.  In this paper we introduce tools for computing twisted THH, which we apply to computations for Thom spectra, Eilenberg-MacLane spectra, and the real bordism spectrum $MU_{\mathbb{R}}$.  In particular, we construct an equivariant version of the Bökstedt spectral sequence, the formulation of which requires further development of the Hochschild homology of Green functors, first introduced by Blumberg, Gerhardt, Hill, and Lawson.
\end{abstract}

\section{Introduction}\label{sec:intro}

 The trace method approach to algebraic $K$-theory uses topological versions of classical constructions from homological algebra to approximate $K$-theory. In recent years this approach has been instrumental in many important algebraic $K$-theory calculations. Topological Hochschild homology (THH) plays a key role in trace methods. Indeed, understanding THH is essential for defining topological cyclic homology (see \cite{BHM} or \cite{NikolausScholze}), which often approximates algebraic $K$-theory quite closely. 
 
Building on the work of Hill, Hopkins, and Ravenel on norms in equivariant homotopy theory \cite{HHR}, the topological Hochschild homology of a ring spectrum $R$ can be viewed as the norm $N_e^{S^1}R$ (see \cite{AnBlGeHiLaMa} and \cite{BrunDundasStolz}). This viewpoint leads to several natural generalizations. In particular, for a $C_n$-equivariant ring spectrum $R$, one can consider the norm $N_{C_n}^{S^1}R$, which is the $C_n$-relative, or twisted, topological Hochschild homology of $R$, as defined in \cite{AnBlGeHiLaMa}. The norm $N_{C_n}^{S^1}R$ can be explicitly constructed  as a twisted cyclic bar construction. While the foundations for twisted THH of equivariant spectra are laid out in \cite{AnBlGeHiLaMa}, no computations have appeared in the literature for this new theory. The primary goal of this article is to develop computational tools for twisted topological Hochschild homology and to study key examples.

One of the foundational tools for computing ordinary topological Hochschild homology is the B\"okstedt spectral sequence. For a field $k$ and a ring spectrum $R$, this spectral sequence has the form
\[
E^2_{*,*} = \HH_*^k(H_*(R;k)) \Rightarrow H_*(\THH(R);k). 
\]
It is natural  to ask whether computations of relative topological Hochschild homology can be approached via an analogous spectral sequence. In this paper we construct an equivariant analogue of the B\"okstedt spectral sequence, converging to the equivariant homology of twisted THH. In the statement below, $\m{E}_{\scalestar}(R)$ denotes the $RO(G)$-graded commutative Green functor $ \m{\pi}_{\scalestar}(E \wedge R)$.

\begin{thm}
	Let $G \subset S^1$ be a finite subgroup and $g$ a generator of $G$. Let $R$ be a $G$-ring spectrum and $E$ a commutative $G$-ring spectrum such that $g$ acts trivially on $E$. If $\m{E}_\scalestar(R)$ is flat over $\m{E}_\scalestar$, then there is an equivariant B\"okstedt-type spectral sequence
	$$E_{s,\scalestar}^2=\m{\HH}_{s}^{\m{E}_\scalestar,G}(\m{E}_\scalestar(R))\Rightarrow \m{E}_{s+\scalestar}(i_G^*\THH_G(R))$$
	that converges strongly.
\end{thm}

In the classical B\"okstedt spectral sequence, the $E^2$-term is ordinary Hochschild homology of a $k$-algebra. Here, however, the $E^2$-term is a type of Hochschild homology for Green functors. The basic construction of Hochschild homology for Green functors is due to Blumberg, Gerhardt, Hill, and Lawson  \cite{BlGeHiLa}. In the current paper we extend their work to define a theory of Hochschild homology for $RO(G)$-graded $\m{E}_{\scalestar}$-algebras, which is what appears in the $E^2$-term of the equivariant B\"okstedt spectral sequence above. This is a spectral sequence of Mackey functors; evaluating at $G/e$ recovers a version of the classical B\"okstedt spectral sequence, which computes $\THH$ with coefficients in the twisted bimodule ${}^g R$:
$$\HH^{E_*} _s (E_*(R); {}^g E_*(R) ) \Rightarrow E_{s+*} (\THH(R ; {}^g R)).$$

Using the equivariant B\"okstedt spectral sequence, we compute the $RO(C_2)$-graded equivariant homology of the $C_2$-relative THH of the real bordism spectrum $MU_{\mathbb{R}}$. 
\begin{thm}
The $RO(C_2)$-graded equivariant homology of $\THH_{C_2}(MU_{\mathbb{R}})$ is 
\begin{align*}
\m{H}^{C_2}_{\scalestar}(i^*_{C_2} \THH_{C_2}(MU_{\mathbb{R}}); \m{\mathbb{F}}_2)
& \cong H\m{\mathbb{F}}_{2\scalestar}[b_1, b_2, \ldots]  \otimes_{\mathbb{F}_2}  \Lambda_{\mathbb{F}_{2}}(z_1, z_2, \ldots)
\end{align*}
as an $H\m{\mathbb{F}}_{2\scalestar}$-module. Here $|b_i| = i\rho $ and $|z_i| = 1+i\rho,$    where $\rho$ is the regular representation of $C_2$. 
\end{thm}

This calculation requires understanding the algebraic structure in the equivariant B\"okstedt spectral sequence, which can be formulated as follows.

\begin{thm}
 If $R$ is endowed with the structure of a commutative $G$-ring spectrum, then the equivariant B\"okstedt spectral sequence inherits the structure of a spectral sequence of $RO(G)$-graded algebras over $\m E_\scalestar$.
\end{thm}

For classical THH, work of Blumberg, Cohen, and Schlichtkrull facilitates the study of topological Hochschild homology of Thom spectra. In \cite{BCS}, they showed that the Thom spectrum functor and the cyclic bar construction ``commute" in a suitable sense, due to nice symmetric monoidal properties of the Thom spectrum functor and its behavior under colimits. They could then easily compute the topological Hochschild homology of various Thom spectra, such as $H\mathbb{F}_p$ and $H\mathbb{Z}$, $MO$, $MU$, and other cobordism spectra. In the current paper, we study (twisted) topological Hochschild homology of equivariant Thom spectra. 

Recent results due to Behrens and Wilson \cite{BW} and Hahn and Wilson \cite{HW} show that certain equivariant Eilenberg--MacLane spectra can be constructed as equivariant Thom spectra.  As proven in \cite{HKZ}, the equivariant Thom spectrum functor is appropriately $G$-symmetric monoidal and commutes with $G$-colimits, which enables us to describe the $\THH$ and twisted $\THH$ of equivariant Thom spectra, and to make the following computations.

\begin{thm} As $C_2$-spectra,
$$\THH_{C_2}(H\underline{\mathbb{F}}_2) \simeq H\underline{\mathbb{F}}_2 \wedge (\bigvee_{k \geq 0} S^{2k\rho} \vee \bigvee_{k \geq 0} S^{2k\rho +2})$$
and
$$\THH_{C_2}(H\underline{\mathbb{Z}}_{(2)}) \simeq H\underline{\mathbb{Z}}_{(2)} \wedge \Omega^{2\sigma} (\mathbb{H}P^\infty \langle 2\sigma + 2 \rangle) _+.$$
Here $\sigma$ denotes the sign representation of $C_2$, and $\rho = \sigma + 1$ denotes its regular representation.
\end{thm}

In \cite{BlGeHiLa}, the authors define a relative version of Hochschild homology for Green functors as well. For $H \subset G$, and $\m{R}$ an $H$-Green functor, the $H$-relative Hochschild homology of $\m{R}$, $\m{\HH}_H^G(\m{R})_*$, is defined using a $G$-twisted cyclic bar construction on the Mackey functor norm, $N_H^G\m{R}$ (see Section \ref{Hochschild} for details). A Green functor for the trivial group is just a classical ring, and hence the relative theory of Hochschild homology for Green functors also yields new ring invariants. For a ring $R$, $\m{\HH}_e^G({R})_*$ is defined using a $G$-twisted cyclic bar construction on the Mackey functor norm $N_e^{G}(R).$ 

In \cite{BlGeHiLa} the authors also construct a type of cyclotomic structure on Hochschild homology for Green functors. In the case of rings, this cyclotomic structure provides the framework for the definition of an algebraic analogue of topological restriction (TR) homology, which the authors call $tr(A)$. In this paper, we perform the following new computations of the $C_{p^n}$-twisted Hochschild homology of rings and of this algebraic TR-theory. 

\begin{prop}
For the constant $\{e\}$-Green functor $\Z$,
\[ \m{\HH}_{e}^{C_{p^n}}(\Z)_k = \begin{cases} \m{A}^{C_{p^n}} &: k=0 \\ 0 &: \text{otherwise,} \end{cases}\]
and
\[tr_{k}(\Z;p) = \begin{cases} \Z^{\infty} &: k=0 \\ 0 &: \text{otherwise.}  \end{cases}
\]
Here $\m{A}^{C_{p^n}}$ is the $C_{p^n}$-Burnside Mackey functor. 
\end{prop}

\subsection{Organization}

In Section \ref{sec:relTHH}, we recall the definition of $C_n$-relative, or twisted, topological Hochschild homology for $C_n$-ring spectra. In Section \ref{sec:HHGreen} we review the definitions of Mackey and Green functors, as well as the theory of Hochschild homology for Green functors. We discuss how this theory of Hochschild homology relates to classical and equivariant topological Hochschild homology. We also make new computations of Hochschild homology for Green functors and of an algebraic analogue of topological restriction homology in this section. Section \ref{sec:bokstedt} focuses on the construction, algebraic structure, and applications of the equivariant B\"{o}kstedt spectral sequence. In Section \ref{sec:thomspectra} we explore a different computational approach in the context of the twisted THH of equivariant Thom spectra.  

\subsection{Notation and Conventions}

Throughout, we are working with genuine orthogonal $G$-spectra indexed on a complete universe. We use $*$ to denote integer gradings, $\scalestar$ to denote $RO(G)$-gradings, and $\bullet$ to denote simplicial gradings.

\subsection{Acknowledgements}
 
This paper is one part of the authors' Women in Topology III project. A second part of that project will appear in a separate article \cite{AGHKK}.  We are grateful to the organizers of the Women in Topology III workshop, as well as to the Hausdorff Research Institute for Mathematics, where much of this research was carried out. We are grateful to Mike Hill and Dylan Wilson for many enlightening discussions related to this work. We also thank Christy Hazel, Asaf Horev, Dan Isaksen, Clover May, and Foling Zou for helpful conversations. The authors also thank an anonymous referee for helpful comments and for catching an error in a previous draft. The second author was supported by NSF grant DMS-1810575. The Women in Topology III workshop was supported by NSF grant DMS-1901795, the AWM ADVANCE grant NSF HRD-1500481, and Foundation Compositio Mathematica. 

\section{Twisted topological Hochschild homology of equivariant spectra}\label{sec:relTHH}

In \cite{AnBlGeHiLaMa}, Angeltveit, Blumberg, Gerhardt, Hill, Lawson, and Mandell define a theory of $C_n$-relative, or twisted, topological Hochschild homology for $C_n$-ring spectra. Given a $C_n$-ring spectrum $R$, its $C_n$-relative topological Hochschild homology, THH$_{C_n}(R)$, is a relative norm, $N_{C_n}^{S^1}(R).$ We now recall the explicit construction of this norm in terms of a twisted cyclic bar construction. More details of this construction can be found in \cite{AnBlGeHiLaMa}. 

Let $R$ be an associative orthogonal $C_n$-ring spectrum indexed on the trivial universe $\mathbb{R}^{\infty}$. The \emph{$C_n$-twisted cyclic bar construction} on $R$, $B_{\bullet}^{cy, C_n}(R)$, is a simplicial spectrum, where
$$
B_{q}^{cy, C_n}(R) = R^{\wedge(q+1)}.
$$
Let $g$ denote the generator $e^{2\pi i/n}$ of $C_n$, and let $\alpha_q: R^{\wedge(q+1)} \to R^{\wedge(q+1)}$ denote the map that cyclically permutes the last factor to the front and then acts on the new first factor by $g$. Let $\mu$ and $\eta$ denote the multiplication map and unit map of $R$, respectively. The face and degeneracy maps on $B_{q}^{cy, C_n}(R)$ are given by 
$$
d_i = \left\{\begin{array}{ll} \Id^{\wedge i} \wedge \mu \wedge \Id^{\wedge (q-i-1)} &: 0\leq i<q \\  (\mu \wedge \Id^{\wedge (q-1)}) \circ \alpha_q &: i=q,
\end{array}\right.
$$
and 
$$
s_i =  \Id^{\wedge (i+1)} \wedge \eta \wedge \Id^{\wedge (q-i)} \hspace{1cm}\forall\,  0\leq i \leq q.
$$
This yields a simplicial object $B_{\bullet}^{cy, C_n}(R)$. 

Recall that there is an operator $\tau_q$ on the $q$-th simplicial level of the classical cyclic bar construction such that $\tau_q^{q+1} = \Id$. This operator $\tau$ satisfies certain relations with the face and degeneracy maps, giving the cyclic bar construction the structure of a cyclic set, which implies that the geometric realization admits an $S^1$-action (\cite{Con}, \cite{BuFi}, \cite{Jon}). The $C_n$-twisted cyclic bar construction, however, is not a cyclic object. Note, though, that for every $q$, the operator $\alpha_q$, satisfies $\alpha_q^{n(q+1)} = \Id.$ Furthermore, the operator $\alpha_q$ satisfies the following relations with respect to the face and degeneracy maps. 
$$
\begin{array}{ll}
d_o \alpha_q = d_q & \\

d_i \alpha_q = \alpha_{q-1}d_{i-1} & 1 \leq i \leq q \\
s_i \alpha_q = \alpha_{q+1}s_{i-1} & 1 \leq i \leq q\\
 s_0 \alpha_q = \alpha_{q+1}^2 s_q & \\
\end{array}
$$
The $C_n$-twisted cyclic bar construction thus admits the structure of a $\Lambda_n^{op}$-object, in the sense of B\"okstedt-Hsiang-Madsen \cite{BHM}. By \cite{BHM} the geometric realization of the $C_n$-twisted cyclic bar construction therefore has an $S^1$-action, extending  the simplicial $C_n$-action generated on the $q$-th level by $\alpha_q^{q+1}$.

The norm from $C_n$ to $S^1$, which is a $C_n$-relative version of topological Hochschild homology, is defined in terms of this twisted cyclic bar construction. 

\begin{defn}
Let $U$ be a complete $S^1$-universe, and let $\widetilde{U} = \iota_{C_n}^* U$, the pullback of the universe to $C_n$. Let $R$ be an associative orthogonal $C_n$-ring spectrum indexed on $\widetilde{U}$. The \emph{$C_n$-relative topological Hochschild homology} of $R$ is defined to be the norm $N_{C_n}^{S^1}(R)$, given by
$$
\THH_{C_n}(R) =N_{C_n}^{S^1}(R) =  \mathcal{I}_{\mathbb{R}^{\infty}}^U |B^{cy, C_n}_{\bullet}(\mathcal{I}_{\widetilde{U}}^{\mathbb{R}^{\infty}} R)|,
$$
where $\mathcal I$ denotes a change-of-universe functor.
\end{defn}

In \cite{AnBlGeHiLaMa} the authors prove that when $R$ is a commutative $C_n$-ring spectrum, the norm functor from commutative $C_n$-ring spectra to commutative $S^1$-ring spectra is left adjoint to the forgetful functor, as one would expect from a norm construction.

In the nonequivariant case,  THH of any commutative ring spectrum $A$  can be constructed in terms of the natural simplicial tensoring of commutative ring spectra over simplicial sets, as $A \otimes S^1$ \cite{MSV}. Analogously, the $C_n$-relative THH of a commutative $C_n$-spectrum $R$ is a relative tensor
$$
\THH_{C_n}(R)\simeq \mathcal{I}_{\mathbb{R}^{\infty}}^U(R\otimes_{C_n}S^1),
$$
which is constructed as follows. If $\widetilde{U}= \iota_{C_n}^* U$, then $R\otimes_{C_n}S^1$ is the coequalizer
$$
\xymatrix{ (\mathcal{I}_{\widetilde{U}}^{\mathbb{R}^{\infty}} R) \otimes C_n \otimes S^1 \ar@<1ex>[r] \ar@<-.5ex>[r] & (\mathcal{I}_{\widetilde{U}}^{\mathbb{R}^{\infty}} R) \otimes S^1 }
,
$$
where one map comes from the usual action of $C_n$ on $S^1$ and the other from the induced action $(\mathcal{I}_{\widetilde{U}}^{\mathbb{R}^{\infty}} R) \otimes C_n \to \mathcal{I}_{\widetilde{U}}^{\mathbb{R}^{\infty}} R.$

Ordinary topological Hochschild homology of a ring spectrum admits a cyclotomic structure (see \cite{HesselholtMadsen} or \cite{NikolausScholze}  for more on cyclotomic spectra). This cyclotomic structure on THH yields maps
\[
R: \THH(A)^{C_{p^{n}}} \to \THH(A)^{C_{p^{n-1}}}
\]
called restriction maps, which can be used to define topological restriction homology:
\[
\textup{TR}(A;p):= \holim_{\substack{\longleftarrow \\ R}} \THH(A)^{C_{p^{n}}}.
\]
Topological restriction homology gives rise in turn to  topological cyclic homology, which is a close approximation to algebraic $K$-theory under good circumstances. Understanding the invariants THH and TC is key to understanding algebraic $K$-theory via the trace method approach.

Cyclotomic structures arise in the equivariant setting as well. In particular, in \cite{AnBlGeHiLaMa} the authors prove that if $R$ is a $C_n$-ring spectrum, and $p$ is prime to $n$, then $\THH_{C_n}(R)$ is a $p$-cyclotomic spectrum. It is thus possible to define $C_n$-relative topological restriction homology and topological cyclic homology, TR$_{C_n}(R;p)$ and TC$_{C_n}(R;p)$.

\section{Hochschild homology for Green functors}\label{sec:HHGreen}

Topological Hochschild homology is a topological analogue of the classical algebraic theory of Hochschild homology. Indeed, topological Hochschild homology can be constructed via a cyclic bar construction, directly modeled after the algebraic construction. For a ring $A$, the topological Hochschild homology of the Eilenberg-MacLane spectrum $HA$ and the Hochschild homology of $A$ are related via a linearization map
$$
\pi_q\left(\THH(HA)\right ) \to \HH_q(A),
$$
which factors the Dennis trace map from algebraic $K$-theory to Hochschild homology
$$
K_q(A) \to \pi_q\left(\THH(HA)\right) \to \HH_q(A). 
$$

In \cite{BlGeHiLa} the authors addressed the natural question of whether there is an algebraic analogue of $C_n$-relative topological Hochschild homology. Recall that $C_n$-relative THH takes as input $C_n$-ring spectra. For $C_n$-ring spectra to arise as Eilenberg-MacLane spectra, the construction of Eilenberg-MacLane spectra must be extended from abelian groups and ordinary rings to Mackey functors and Green functors. The algebraic analogue of $C_n$-relative topological Hochschild homology should thus be a theory of Hochschild homology for Green functors. We recall this theory in Section \ref{Hochschild} below, after reviewing  basic definitions for Mackey and Green functors.

\subsection{Mackey and Green functors}\label{Mackey}
Throughout this section, let $G$ denote a finite abelian group.

In equivariant homotopy theory, Mackey functors play the role that abelian groups play in the non-equivariant theory. In particular, a $G$-spectrum $X$ has associated homotopy Mackey functors, rather than just homotopy groups. We begin by recalling the definition of a Mackey functor. 

\begin{defn}
Given finite $G$-sets $S$ and $T$, a \emph{span} from $S$ to $T$ is a diagram
$$
S \leftarrow U \to T,
$$
where $U$ is also a finite $G$-set, and the maps are $G$-equivariant. An \emph{isomorphism of spans} is a commuting diagram of finite $G$-sets
$$
\xymatrix @R=.6pc{& U \ar[dl] \ar[dd]^{\cong} \ar[dr] & \\
S & & T. \\
& V \ar[ul] \ar[ur] & \\
}
$$
Composition of spans is given by pullback. There is also a monoidal product on the set of spans with fixed endpoints. Given two spans, $S \leftarrow U \to T$ and $S \leftarrow U' \to T$, their product is defined via the disjoint union, $S\leftarrow U \amalg U'\to T$. 
\end{defn}

\begin{defn}
The \emph{Burnside category} of $G$, $\mathcal{A}_G$, has as objects all finite $G$-sets. For finite $G$-sets $S$ and $T$, the morphism set $\mathcal{A}_G(S, T)$ is the group completion of the monoid of isomorphism classes of spans $S \leftarrow U \to T$.
\end{defn}

\begin{defn}
A \emph{Mackey functor} is an additive functor
$
\m{M}: \mathcal{A}_G^{op} \to {\mathcal Ab}.
$
\end{defn}

Note that any finite $G$-set $X$ is isomorphic to a disjoint union of orbits $G/H$ for various subgroups $H$ of $G$. It follows that every Mackey functor is determined by its values on the orbits $G/H$, since Mackey functors are additive. 

Explicitly, a Mackey functor $\m{M}$ is equivalent to a pair of functors from finite $G$-sets to abelian groups
$$
M_*, M^*: GSets \to {\mathcal Ab},
$$
where $M_*$ is covariant and $M^*$ contravariant, such that both functors take disjoint unions to direct sums, and the following conditions are satisfied. For any $G$-set $X$,  $M_*(X) = M^*(X)$, and the common value is denoted $\m{M}(X).$  Further, if
$$
\xymatrix{W \ar[r]^{g'} \ar[d]^{f'} & X \ar[d]^{f} \\
Z \ar[r]^{g} & Y
}
$$
is a pullback diagram in $GSets$, then $M^*(f)M_*(g) = M_*(g')M^*(f')$. 

Every sequence of subgroup inclusions $K \subset H \subset G$ induces a natural surjection $q_{K,H}: G/K \to G/H$. The homomorphism $M_*(q_{K,H}):\m M(G/K) \to \m M (G/H)$ is called the \emph{transfer map} and denoted $tr_K^H.$ The homomorphism $M^*(q_{K,H}):\m M(G/H) \to \m M (G/K)$ is called the \emph{restriction map} and denoted $res_K^H$.

\begin{defn}
The \emph{Burnside ring of $G$}, $A(G)$, is the group completion of the monoid of isomorphism classes of finite $G$-sets under disjoint union. Multiplication in this ring is given by Cartesian product. 
\end{defn}

\begin{exmp}\label{ex:Burnside}The \emph{Burnside Mackey functor} for $G$, denoted $\m{A}$, is defined by $\m{A}(G/H) = A(H)$ for all $H\subset G$. The transfer and restriction maps are given by induction and restriction maps on finite sets. More explicitly, for $K \subset H \subset G$,  and $X$ a finite $K$-set and $Y$ a finite $H$-set,
$$tr_K^H([X]) = [H \times_K X] \quad \text{and} \quad res_K^H([Y]) = [i_K^H(Y)],$$ 
where $i_K^H: HSet \to KSet$ is the restriction functor.
\end{exmp}

\begin{exmp}
Let $X$ be a $G$-spectrum. For all $q$, the equivariant homotopy groups of $X$ form a $G$-Mackey functor, $\m{\pi}^G_q$, defined by
$$
\m{\pi}^G_q(X)(G/H) := \pi_q(X^H).
$$
\end{exmp}

The category Mack$_G$  of $G$-Mackey functors admits a symmetric monoidal structure defined as follows.

\begin{defn} The \emph {box product} $\m{L} \Box \m{M}$ of two $G$-Mackey functors $\m{L}$ and $\m{M}$  is defined as a left Kan extension over the Cartesian product of finite $G$-sets 
\[
\xymatrix{ {\mathcal{A}_G\times\cA_G}\ar[d]_{\mhyphen\times\mhyphen} \ar[rr]^{\m{L}\times \m{M}} &&
  {\mathcal Ab\times\mathcal Ab}\ar[rr]^{\mhyphen\otimes\mhyphen} && {\mathcal Ab}
  \\ {\mathcal{A}_G}\ar[urrrr]_{\m{L}\Box\m{M}} }
  \]
and is again a $G$-Mackey functor. 
\end{defn}

It is easy to see that the Burnside Mackey functor $\m{A}$ is a unit for the box product.

The box product in Mackey functors is closely related to the smash product of orthogonal $G$-spectra. It follows from \cite[1.3]{LeMa} that for cofibrant $(-1)$-connected $G$-spectra $X$ and $Y$, there is a natural isomorphism
\begin{equation}\label{compatibility}
\m{\pi}_0X \Box \m{\pi}_0Y \cong \m{\pi}_0(X \wedge Y). 
\end{equation}

 A $G$-Mackey functor $\m{M}$ has an associated Eilenberg-Mac Lane $G$-spectrum, $H\m{M}$, the defining property of which is that   
$$
\m{\pi}^G_k(H\m{M}) \cong \left\{ \begin{array}{ll} \m{M} &: k = 0 \\
0 &: k \neq 0, \\
\end{array} \right.
$$
(see, for example, \cite{dS} or \cite{dSNie}).
We can then give a homotopical description of the box product of Mackey functors. It follows from isomorphism \ref{compatibility} above that for any two $G$-Mackey functors $\m{L}$ and $\m{M}$:
$$
\m{L} \Box \m{M} \cong \m{\pi}_0(H\m{L} \wedge 
H\m{M}).
$$

\begin{defn}
A \emph{Green functor} is an associative monoid in the symmetric monoidal category Mack$_G$. A  \emph{commutative Green functor} is a commutative monoid.
\end{defn}

In this article, we also need \emph{graded} Mackey functors, both $\mathbb{Z}$-graded and $RO(G)$-graded, where $RO(G)$ denotes the real representation ring of $G$.
 
\begin{defn}
(1) A \emph{$\mathbb{Z}$-graded Mackey functor} for $G$, $\m{M}_*$, is a functor from the discrete category $\mathbb Z$ to Mack$_G$, i.e.,   a set $\{\m{M}_q\mid q\in \mathbb Z\}$ of  Mackey functors. A map of $\mathbb{Z}$-graded Mackey functors for $G$, $\m{L}_* \to \m{M}_*$, is a natural transformation, i.e.,  a set $\{\m{L}_q \to \m{M}_q\mid q\in \mathbb Z\}$ of  maps of Mackey functors. 

(2) An \emph{$RO(G)$-graded Mackey functor} for $G$, $\m{M}_{\scalestar}$, is a functor from the discrete category $RO(G)$ to Mack$_G$, i.e.,   a set $\{\m{M}_\alpha\mid \alpha\in RO(G)\}$ of  Mackey functors. A map of $RO(G)$-graded Mackey functors, $\m{L}_\scalestar \to \m{M}_\scalestar$, is a natural transformation, i.e.,  a set $\{\m{L}_\alpha \to \m{M}_\alpha\mid \alpha\in RO(G)\}$ of  maps of Mackey functors.
\end{defn}

\begin{exmp}
For every $G$-spectrum $X$, there is a  $\mathbb{Z}$-graded homotopy Mackey functor, $\m{\pi}_*^G(X)$, given by $\m{\pi}_q^G(X)(G/H) = \pi_q(X^H)$ for all $q\in \mathbb Z$.

\end{exmp}
Recall that for a $G$-spectrum, $X$, one can define $RO(G)$-graded equivariant homotopy groups as follows. An element $\alpha \in RO(G)$ can be written as $ \alpha = [\beta] - [\gamma]$, where $\beta$ and $\gamma$ are finite dimensional real representations of $G$. Let $H$ be a subgroup of $G$. Then the  equivariant homotopy group $\pi_{\alpha}(X^H)$ is defined as
$$
\pi_{\alpha}(X^H) := [S^{\beta} \wedge G/H_{+}, S^{\gamma} \wedge X ]_G
$$
Note that this is a priori only well defined up to non-canonical isomorphism. In order to get around this, one can work with $hRO(G,U)$, the homotopy category of a certain category of representations of $G$ embedded in a chosen complete $G$-universe $U$. See Section XIII.1 of \cite{PMay} for details. This ensures that we account for a Burnside ring's worth of automorphisms of each $\alpha \in RO(G)$. 

\begin{exmp} For a $G$-spectrum $X$, there is an $RO(G)$-graded homotopy Mackey functor, $\m{\pi}_{\scalestar}^G(X)$, given by  $\m{\pi}_{\alpha}^G(X)(G/H) = \pi_{\alpha}(X^H)$ for all $\alpha\in RO(G)$.
\end{exmp}

As discussed in \cite{LeMa}, there is a graded version of the box product, given by Day convolution, endowing the categories of $\mathbb{Z}$-graded and $RO(G)$-graded Mackey functors with symmetric monoidal structures. 

\begin{defn}
Let $\m{L}_*$ and $\m{M}_*$ be $\mathbb{Z}$-graded Mackey functors for $G$. The graded box product $\m{L}_* \Box \m{M}_*$ is defined in terms of the ungraded box product by
$$
(\m{L}_* \Box \m{M}_*)_q = \bigoplus_{i+j = q} \m{L}_i \Box \m{M}_j.
$$
Similarly,  the graded box product $\m{L}_{\scalestar} \Box \m{M}_{\scalestar}$  of $RO(G)$-graded Mackey functors $\m{L}_{\scalestar}$ and $\m{M}_{\scalestar}$ is given by
$$
(\m{L}_{\scalestar} \Box \m{M}_{\scalestar})_{\alpha} = \bigoplus_{\gamma + \beta = \alpha} \m{L}_{\gamma} \Box \m{M}_{\beta}.
$$
\end{defn}

The unit for the ($\mathbb{Z}$ or $RO(G)$)-graded box product is the ($\mathbb{Z}$ or $RO(G)$)-graded Burnside Mackey functor $\m{A}_*,$ which is the Burnside Mackey functor $\m{A}$ in degree 0, and 0 in all other degrees. 

It is important in this paper to know that the $RO(G)$-graded homotopy functor is monoidal.

\begin{lem}\label{lem:monoidal}\cite [Theorem 5.1]{LeMa} For every finite group $G$, the $RO(G)$-graded homotopy functor $\m \pi_\scalestar$ from $G$-spectra to $RO(G)$-graded Mackey functors is monoidal.
\end{lem}

In this work we will also need to consider graded Green functors. We recall their definition from \cite{LeMa}, where they are referred to as graded Mackey functor rings.

\begin{defn}
A \emph{graded Green functor} (with a $\mathbb{Z}$ or $RO(G)$-grading) is an associative monoid in the category of graded Mackey functors, with respect to the graded box product.
\end{defn}

Lemma \ref{lem:monoidal} implies that if $R$ is a $G$-ring spectrum, then $\m \pi_\scalestar(R)$ is an $RO(G)$-graded Green functor.

\subsection{Hochschild homology for Green functors}\label{Hochschild}

In this section we recall the construction of Hochschild homology for Green functors, given by Blumberg, Gerhardt, Hill, and Lawson in \cite{BlGeHiLa}. We then extend this construction to  graded Green functors.

We begin by observing that if $G$ is cyclic, then every $G$-Mackey functor $\m{M}$ admits a natural $G$-action. To see this, recall that for any group $G$ and $G$-Mackey functor $\m{M}$, the Weyl group $W_G(H)= N_G(H)/H$ acts on $\m{M}(G/H)$. If $G$ is a cyclic group, then $W_G(H)=G/H$ for every subgroup $H$ of $G$. It follows that $\m{M}(G/H)$ is a $G$-module for all $H$, and the restriction and transfer maps are maps of $G$-modules, i.e.,  the Mackey functor $\m{M}$ admits a $G$-action. We define the $G$-twisted cyclic bar construction for a $G$-Mackey functor using this action.

\begin{defn}
Let $G \subset S^1$ be a finite subgroup, and let $g$ denote the generator $e^{2\pi i/|G|}$ of $G$. Let $\m{R}$ be a $G$-Green functor. The \emph{$G$-twisted cyclic bar construction on $\m{R}$}, $\m{B}_{\bullet}^{cy, G}(\m{R})$, is a simplicial Green functor, where
$$
\m{B}_{q}^{cy, G}(\m{R}) = \m{R}^{\Box(q+1)}.
$$
The face and degeneracy maps are defined as they are for the twisted cyclic bar construction in Section \ref{sec:relTHH} above. 
\end{defn}

More generally, one can define the twisted cyclic nerve of a $G$-Green functor $\m{R}$ with coefficients in an $\m{R}$-bimodule $\m{M}$, with respect to  an element $g \in G$. First, we explain how to twist module structures over Green functors.

\begin{defn}\label{twistaction}
Let $G \subset S^1$ be a finite subgroup, and let $g \in G$. Let $\m{R}$ be a Green functor for $G$, and let $\m{M}$ be a left
$\m{R}$-module with action map $\lambda$. The \emph{$g$-twisted module structure} on $\m M$, denoted $\gmM$, has  action map ${}^{g}\!\lambda$ specified by the commuting diagram
$$
\xymatrix{\m{R} \Box \m{M} \ar[d]_{g \Box 1} \ar[dr]^{{}^{g}\!\lambda} & \\
\m{R} \Box \m{M} \ar[r]^{\lambda} & \m{M}.}
$$
\end{defn}

\begin{defn}
Let $G \subset S^1$ be a finite subgroup, and let $g \in G$. Let $\m{R}$ be a Green functor for $G$, and let $\m{M}$ be an
$\m{R}$-bimodule. The \emph{$G$-twisted cyclic nerve},  $\m{B}_{\bullet}^{cy,G}(\m{R}; \gmM)$, is the simplicial Mackey functor with $q$ simplices
$$
\m{B}_{q}^{cy, G}(\m{R}, \gmM) =\gmM \Box \m{R}^{\Box q}.
$$
The face maps $d_i$ are given as usual by multiplication of the $i$th and $(i+1)$st factors if $0< i <q$. The face map $d_0$ is the ordinary right module action map for $\m{M}$, while the last face map, $d_q$, rotates the last factor to the front and then uses the twisted left action map of Definition \ref{twistaction}. The degeneracy maps $s_i$ are induced by inclusion of the unit after the $i$th factor, for $0 \leq i \leq q$.
\end{defn}

The twisted cyclic bar construction above is the case where the bimodule $\m{M}$ is $\m{R}$ itself, and $g$ is the generator $e^{2\pi i/|G|}$ of $G$, i.e., 
$$
\m{B}_{\bullet}^{cy, G}(\m{R}) = \m{B}_{\bullet}^{cy,G}(\m{R}; \gmR).
$$

We can now define the $G$-twisted Hochschild homology of $\m{R}$ by taking the homology of this simplicial object. 
\begin{defn}
Let $G \subset S^1$ be a finite group and $\m{R}$ an associative Green functor for $G$. The \emph{$G$-twisted Hochschild homology} of $\m{R}$ is defined by
$$
\m{\HH}_i^G(\m{R}) = H_i(\m{B}_{\bullet}^{cy, G}(\m{R})).
$$
\end{defn}
Applying the Dold-Kan correspondence at each orbit gives rise to an equivalence between the category of simplicial Mackey functors and the category of non-negatively graded dg Mackey functors. By definition, the homology of a simplicial Mackey functor is the homology of the associated normalized dg Mackey functor. See \cite{BlGeHiLa} for more discussion of simplicial and differential graded Mackey functors.

One motivation for developing  a Hochschild theory for Green functors is the desire for an algebraic analogue of the twisted topological Hochschild homology discussed in Section \ref{sec:relTHH}. Let $H \subset G \subset S^1$ be finite subgroups and $R$ an $H$-ring spectrum. As seen in \cite{BlGeHiLa}, there is a weak equivalence 
$$
\iota_G^* \THH_H(R) = N_H^G R \wedge_{N_H^G R^e} {}^{g}\!N_H^G R,
$$
where $N_H^G R^e$ denotes the enveloping algebra $N_H^G R \wedge (N_H^G R)^{op}$, and ${}^{g}\!N_H^G R$ denotes $N_H^GR$ with the bimodule structure twisted on one side by the Weyl action. A full algebraic analogue of relative THH therefore requires a notion of $G$-twisted Hochschild homology for an $H$-Green functor $\m{R}$, defined using norms on Mackey functors. We now recall the definition of Mackey functor norms from \cite{HillHopkins16}.

\begin{defn}
Let $H$ be a subgroup of a finite group $G$, and let $\m{M}$ be an $H$-Mackey functor. The \emph{norm} $N_H^G \m{M}$ is defined to be the $G$-Mackey functor
$$
N_H^G \m{M} := \m{\pi}_0^G N_H^G( H\m{M}).
$$
\end{defn}

\begin{defn}
Let $H \subset G \subset S^1$ be finite subgroups, and let $\m{R}$ be an associative Green functor for $H$. The \emph{$G$-twisted Hochschild homology} of $\m{R}$ is defined to be
$$
\m{\HH}^G_H(\m{R})_i = H_i(\m{B}_{\bullet}^{cy, G}(N_H^G\m{R})).
$$
\end{defn}

The Hochschild homology for Green functors defined in \cite{BlGeHiLa} and recalled above is an equivariant analogue of taking classical Hochschild homology of a ring (i.e., a $\mathbb{Z}$-algebra). In this article, we also need an analogue of Hochschild homology of an associative $k$-algebra, for $k$ any commutative ring.

To define this analogue, we need the relative box product \cite{LeMa}, which is a particular case of the general construction of a relative tensor product in a monoidal category. Let $\m{R}$ be a commutative Green functor, and let $\m{L}$ and $\m{M}$ be left and right $\m{R}$-modules respectively. The Mackey functor $\m{L} \Box_{\m{R}} \m{M}$ is defined to be the coequalizer
$$\xymatrix{
\m{L} \Box \m{R} \Box \m{M} \ar@<1ex>[r]^{\rho \Box id}\ar@<-.5ex>[r]_{id \Box \lambda}  & \m{L} \Box \m{M} \ar[r] &\m{L} \Box_{\m{R}} \m{M}.  }
$$
Here $\rho$ is the right action map for $\m{L}$ and $\lambda$ the left action map for $\m{M}$. Lewis and Mandell prove in \cite{LeMa} that the category of left $\m{R}$-modules, Mod$_{\m{R}}$, is a closed symmetric monoidal abelian category with monoidal product $\Box_{\m{R}}$. The Green functor $\m{R}$ is commutative, so left modules become bimodules by using the same action on the right. We will call a monoid in this category an \emph{associative $\m{R}$-algebra}. We define Hochschild homology for associative $\m{R}$-algebras as above, but replacing $\Box$ with $\Box_{\m{R}}$. For $\m{T}$ an associative $\m{R}$-algebra for $G$, we denote its $G$-twisted Hochschild homology by $\m{\HH}^{\m{R},G}_*(\m{T})$. 

\subsection{$C_{p^n}$-twisted Hochschild homology of $\mathbb{Z}$}

Since a Green functor for the trivial group is just a ring, the theory of Hochschild homology for Green functors described above also yields new ring invariants. In \cite{BlGeHiLa}, the authors computed the $C_{p^n}$-twisted Hochschild homology of the ring ${\F}_p$.  Here we provide an additional explicit example, computing the $C_{p^n}$-twisted Hochschild homology of $\Z$.

\begin{exmp}\label{ex:compZ}
The ring $\mathbb{Z}$, considered as a Green functor for the trivial group, is the Burnside Mackey functor for the trivial group. The norm from $H$ to $G$ of the $H$-Burnside Mackey functor is the $G$-Burnside Mackey functor \cite{MThesis}, so $N_e^{C_p}{\Z}$ is the $C_p$-Burnside functor, $\m{A}^{C_p}$.

Since the Burnside functor is a unit for the box product, $\left(\m{A}^{C_p}\right)^{\square(q+1)}=\m{A}^{C_p}$ for all $q$, whence the
$C_p$-twisted cyclic bar construction on $\m{A}^{C_p}$ is $\m{A}^{C_p}$ in each degree.
The action of the Weyl group is trivial, so each face map is an isomorphism. Calculating the homology of the complex
\[ 0 \gets  \m{A}^{C_p} \xleftarrow{0} \m{A}^{C_p} \xleftarrow{\cong}  \m{A}^{C_p} \xleftarrow{0} \m{A}^{C_p} \xleftarrow{\cong} \m{A}^{C_p} \gets \cdots ,
 \]
we obtain
\[ \m{\HH}^{C_p}_{e}(\Z)_i = \begin{cases} \m{A}^{C_p} &: i=0 \\ 0 &: \text{otherwise.}\end{cases} \]

Likewise, for any $C_n \subset S^1$, $N_e^{C_{n}}\Z$ is the $C_{n}$-Burnside Mackey functor, $\m{A}^{C_{n}}$. Then, the same argument as above shows
\[ \m{\HH}^{C_{n}}_{e}(\Z)_i = \begin{cases} \m{A}^{C_{n}} &: i=0 \\ 0 &: \text{otherwise.}\end{cases} \]

\end{exmp}

\subsection{Relationship to classical and twisted THH}

As mentioned above, one motivation for studying Hochschild homology for Green functors is that it serves as an algebraic analogue of twisted topological Hochschild homology of equivariant spectra. Indeed, Blumberg, Gerhardt, Hill, and Lawson \cite{BlGeHiLa} prove that for $H \subset G \subset S^1$ finite subgroups and $R$ a (-1)-connected commutative $H$-ring spectrum, there is a linearization map
$$
\m{\pi}_k^G\THH_H(R) \to \m{\HH}^G_H(\m{\pi}_0^HR)_k,
$$
which is an isomorphism in degree zero. This is analogous to the classical linearization map from topological Hochschild homology to Hochschild homology in the non-equivariant setting. In Section \ref{sec:bokstedt}, we show that the relationship between relative THH and Hochschild homology for Green functors goes even deeper, by producing an equivariant version of the B\"okstedt spectral sequence. 

When $H$ is the trivial group, the linearization map above yields new trace maps from algebraic $K$-theory. Indeed, in \cite{BlGeHiLa} the authors prove that for $A$ a commutative ring and $G\subset S^1$ a finite group, there is a trace map
$$
K_q(A) \to \m{\HH}_e^G(A)(G/G)_q,
$$
which lifts the classical Dennis trace
$$
K_q(A) \to \HH_q(A). 
$$
The trace map $K_q(A) \to \m{\HH}_e^G(A)(G/G)_q$ factors through $\m{\pi}_q^G\THH(A)(G/G) \cong \pi_q(\THH(A)^G).$ As discussed in Section \ref{sec:relTHH}, the cyclotomic structure on THH yields maps
\[
R: \THH(A)^{C_{p^{n}}} \to \THH(A)^{C_{p^{n-1}}}
\]
called restriction maps, which can be used to define topological restriction homology: 
\[
\textup{TR}(A;p):= \holim_{\substack{\longleftarrow \\ R}} \THH(A)^{C_{p^{n}}}.
\]
Topological restriction homology gives rise in turn to topological cyclic homology. 

It is natural to ask whether the algebraic theory $\m{\HH}_e^G(A)(G/G)_q$ admits structure maps analogous to the restriction maps on THH. In \cite{BlGeHiLa}, the authors proved that Hochschild homology for Green functors does have a type of cyclotomic structure, which in particular yields maps
\[
r: \m{\HH}_e^{C_{p^n}}(A)_q(C_{p^n}/C_{p^n}) \to \m{\HH}_e^{C_{p^{n-1}}}(A)_q(C_{p^{n-1}}/C_{p^{n-1}}),
\]
the \emph{algebraic restriction maps}. One can then define an algebraic analogue of TR-theory, $tr$, given in degree $k$ by
\[
tr_k(A;p) := \lim_{\substack{\longleftarrow\\r}} \m{\HH}_{e}^{C_{p^{n}}}(A)_k(C_{p^{n}}/C_{p^{n}}).
\]
In \cite{BlGeHiLa} the authors showed that 
\[tr_k(\mathbb{F}_p;p) = \begin{cases} \Z_p &: k=0 \\ 0 &: \text{otherwise.} \end{cases}
\]
Below, we compute a second example, the algebraic TR of $\mathbb{Z}$. 

\begin{prop}
The algebraic $TR$-theory of $\Z$ is given by
\[ tr_{k}(\Z; p)=\begin{cases} \Z^{\infty} &: k=0 \\ 0 &: \text{otherwise.} \end{cases} \]
\end{prop}

\begin{proof} 

The algebraic restriction map,
\[ r:\m{\HH}_e^{C_{p^n}}(\Z)_k(C_{p^n}/C_{p^n}) \to \m{\HH}_e ^{C_{p^{n-1}}}(\Z)_k(C_{p^{n-1}}/C_{p^{n-1}}),\]
is constructed in \cite[Corollary 5.18]{BlGeHiLa} on the chain level as the composite
\begin{align*}
    \m{B}_k^{cy, C_{p^n}}\left(N_e^{C_{p^n}}\Z\right)(C_{p^n}/C_{p^n}) &\to \left(\m{B}_k^{cy, C_{p^n}}\left(N_e^{C_{p^n}}\Z\right) \square \left(\m{A}^{C_{p^n}}/E\mathcal{F}_{C_p} \left(\m{A}^{C_{p^n}}\right) \right)\right)(C_{p^n}/C_{p^n}) \\
    &\to \Phi^{C_p}\left(\m{B}_k^{cy, C_{p^n}}\left(N_e^{C_{p^n}}\Z\right)\right)(C_{p^{n-1}}/C_{p^{n-1}}) \\
    &\to \m{B}_k^{cy, C_{p^{n-1}}}\left(N_e^{C_{p^n}}\Z\right)(C_{p^{n-1}}/C_{p^{n-1}}),
\end{align*}
where ${E}\mathcal{F}_{C_p}(\m{A}^{C_{p^n}})$ is the subMackey functor of $\m{A}^{C_{p^n}}$ generated by $\m{A}(C_{p^n}/H)$ for all $H$ not containing $C_p$ \cite[Definition 5.4]{BlGeHiLa}. The first map is induced by the quotient map of Mackey functors $\m{A} \to \m{A}/E\mathcal{F}_{C_p}\left(\m{A}^{C_{p^n}}\right)$. 
The second and third maps are isomorphisms, due respectively to the definition of the geometric fixed points functor, $\Phi^{C_p}$, and the cyclotomic structure of the twisted Hochshild complex \cite[Definition 5.10, Proposition 5.17]{BlGeHiLa}.

Recall  from the computation of the $C_{p^n}$-twisted Hochschild homology of $\Z$ in Example \ref{ex:compZ} that $\m{B}_{k}^{cy, C_{p^n}}(N_e^{C_{p^n}}\Z)=\m{A}^{C_{p^n}}$ for all $k\geq 0$. As an abelian group, $\m{A}^{C_{p^n}}\left(C_{p^n}/C_{p^k} \right)  \cong \Z^{k+1}$. Here we choose the isomorphism so that the $i$th component of $\Z^{k+1}$ corresponds to the orbit $C_{p^k}/C_{p^{k-i}}$ for $0\leq i\leq k.$ Recall from Example \ref{ex:Burnside} that the transfer maps in the Burnside Mackey functor come from induction maps on finite sets. In the Burnside Mackey functor $\m{A}^{C_{p^n}}$ we have 
\begin{align*}
    tr_{C_{p^{k-1}}}^{C_{p^{k}}}\left(x_1, \ldots, x_k \right)=\left(0, x_1, \ldots, x_k \right).
\end{align*}
In degree zero, the first map in the composition above is the quotient of $\m{A}^{C_{p^n}}(C_{p^n}/C_{p^n})=\Z^{n+1}$ by the transfers from $e$, the only subgroup of $C_{p^n}$ not containing $C_p$ \cite[Definition 5.4]{BlGeHiLa}.  Using the formula for transfers given above, we see the image of these transfers is the last copy of $\Z$. 
So, the map induced on Hochschild homology by the composition is the map $\Z^{n+1} \to \Z^n$ quotienting by the last factor.

Taking the inverse limit of $\m{\HH}_{e}^{C_{p^n}}(\Z)_0$ along these maps, we obtain
\[ tr_{0}(\Z;p) = \Z^{\infty}. \]
For all other $k$, the groups $\m{\HH}_{e}^{C_{p^n}}(\Z)_k(C_{p^n}/C_{p^n})$ are zero and so $tr_{k}(\Z; p)=0$.
\end{proof}

\section{The equivariant B\"okstedt spectral sequence}\label{sec:bokstedt}

One of the key computational tools for classical topological Hochschild homology  is the B\"{o}kstedt spectral sequence. This spectral sequence, arising from the skeletal filtration of the simplicial spectrum $\THH(R)_{\bullet}$, computes the homology of THH($R$). In particular, for a ring spectrum $R$ and a field $k$, the B\"{o}kstedt spectral sequence has the form
$$
E^2_{*,*} = \HH^k_*(H_*(R;k)) \Rightarrow H_*(\THH(R);k),
$$
where $\HH^k_*$ denotes Hochschild homology over the field $k$. In this section, under appropriate flatness conditions we construct an equivariant analogue of the B\"{o}kstedt spectral sequence, which converges to the equivariant homology (integer or $RO(G)$-graded) of $G$-relative THH, for $G$ a finite subgroup of $S^1$:
$$
E_{s,\scalestar}^2=\m{\HH}_{s}^{\m{E}_\scalestar,G}(\m{E}_\scalestar(R))\Rightarrow \m{E}_{s+\scalestar}(i_{G}^*\THH_{G}(R)).
$$
Here $E$ is a commutative $G$-ring spectrum such that the generator $g = e^{2\pi i/|G|}$ of $G$ acts trivially on $E$ (see the discussion before Proposition \ref{prop:action_on_Mackey}), $R$ is a $G$-ring spectrum, and $\m{E}_{\scalestar}(R)$ denotes $\m{\pi}_{\scalestar}(E \wedge R)$. The $E^2$-term of this spectral sequence, which we construct in Section \ref{ss:specsequ},  is the Hochschild homology of an algebra over a graded Green functor. 

If $R$ is a commutative ring spectrum, the classical B\"{o}kstedt spectral sequence is a spectral sequence of algebras \cite{ekmm}. We prove  an analogous result in Section \ref{ss:specsequ} for the equivariant B\"{o}kstedt spectral sequence. Finally in Section \ref{ss:MUR}, we use this new spectral sequence to compute the equivariant homology of $\THH_{C_2}(MU_{\mathbb{R}}),$ the $C_2$-relative THH of the real bordism spectrum. 

To construct the equivariant B\"{o}kstedt spectral sequence, we first need to generalize the theory of Hochshild homology for Green functors. In \cite{BlGeHiLa} this theory is defined for (ungraded) Green functors, while we require a theory of Hochschild homology for integer- and $RO(G)$-graded Green functors, and for algebras over a graded commutative Green functor. We begin by setting up these theories.

\subsection{Graded twisted Hochschild homology}

We explain here how to  extend the work of \cite{BlGeHiLa} to the graded setting, in a manner analogous to the extension of ordinary Hochschild homology from rings to graded rings or even differential graded algebras, based on the framework elaborated in \cite{LeMa}.

Let $G\subset S^1$ be a finite subgroup. Let $\m{R}_*$ be a $\mathbb{Z}$-graded $G$-Green functor and let $\m{M}_*$ be an $\m{R}_*$-bimodule. We define the $G$-twisted cyclic nerve of $\m{R}_*$ with coefficients in ${}^g\m{M}_*$ as in Section \ref{sec:HHGreen}, but with a modification of the sign on the last face map. In particular, let 
$$
\m{B}_{q}^{cy, G}(\m{R}_*; {}^g\m{M}_*) = {}^g\m{M}_* \Box\m{R}_*^{\Box q}.
$$
For $1 \leq i \leq q-1$, the face map $d_i$ multiples the $i$th and $(i+1)$st box factors. The map $d_0$ applies the right action of $\m{R}_*$ on $\m{M}_*$. To define the last face map, $d_q$, we need the following isomorphism, which is analogous to the usual symmetry isomorphism for the tensor product of graded modules over a commutative ring.

\begin{defn}
Let $\m{A}_*$ and  $\m{B}_*$ be two $\mathbb{Z}$-graded Mackey functors.
The \emph{rotating isomorphism} $$\tau: \m{A}_*\sqr\m{B}_*\to \m{B}_*\sqr\m{A}_*$$ is defined in level $k$ to be the composite of the isomorphisms $\m{A}_i \sqr \m{B}_j \cong \m{B}_j \sqr \m{A}_i$ (where $k=i+j$) induced by the symmetry isomorphism of the tensor product of abelian groups, and the automorphism $\m{B}_j \sqr \m{A}_i \to \m{B}_j \sqr \m{A}_i$ induced by multiplication by $(-1)^{ij}$.
\end{defn}

The face map $d_q$ is then defined to be the composite
$$\xymatrix{
{}^g\m{M}_* \Box \m{R}_*^{\Box q} \ar[r]^-{\tau_q} &  \m{R}_* \Box {}^g\m{M}_* \Box \m{R}_*^{\Box(q-1)} \ar[rr]^{{}^g\lambda \Box \Id} && {}^g\m{M}_* \Box \m{R}_*^{\Box(q-1)}, 
}
$$
where the map $\tau_q$ is given by iterating the rotating isomorphism, bringing the last factor to the front. The $G$-twisted Hochschild homology $\m{\HH}^G_i(\m{R}_*; {}^g\m{M}_*)$ is then the homology of $\m{B}_{q}^{cy, G}(\m{R}_*; {}^g\m{M}_*)$, equipped with the face maps defined above. When $\m{M}_*=\m{R}_*$ and $g$ is the generator $e^{2\pi i/|G|}$ of $G$, we write $\m{\HH}^G_i(\m{R}_*) = \m{\HH}^G_i(\m{R}_*; {}^g\m{R}_*)$ This construction can be generalized easily to define $\m{\HH}^G_H(\m{R}_*)_i$ for a graded $H$-Green functor $\m{R}_*$, by incorporating norms as we did earlier.

For $RO(G)$-graded Mackey functors, the situation is more subtle. As above, in the graded case we will pick up an additional ``sign" in the last face map. When the gradings live in the representation ring $RO(G)$, however, these generalized signs are elements of the Burnside ring for $G$.

\begin{defn}
	Let $G\subset S^1$ be a finite subgroup, and let $\alpha, \beta$ be two finite-dimensional, real representations of $G$, with corresponding representation spheres $S^\alpha$ and $S^\beta$.  The switch map $S^\alpha\wedge S^\beta \to S^\beta\wedge S^\alpha$ specifies an element in the Burnside ring $A(G)\cong \pi^G_0(S^0)$, which we denote by $\sigma(\alpha, \beta)$.
\end{defn}

\begin{rem}
	The switch map for $G=C_p$ is computed, for instance, in \cite[Proposition 2.56, 2.58]{shulman2014equivariant}.
\end{rem}

\begin{defn}
Let $\m{A}_\scalestar$ and  $\m{B}_\scalestar$ be $RO(G)$-graded Mackey functors.
The \emph{rotating isomorphism} $\tau: \m{A}_\scalestar\sqr\m{B}_\scalestar\to \m{B}_\scalestar\sqr\m{A}_\scalestar$ is defined in level $\alpha$ to be the composite of the isomorphisms $\m{A}_\beta\sqr\m{B}_\gamma\cong \m{B}_\gamma\sqr\m{A}_\beta$ (where $\alpha = \beta + \gamma$) induced by the symmetry isomorphism of the tensor product of abelian groups, and of the automorphism $\m{B}_\gamma\sqr\m{A}_\beta\to \m{B}_\gamma\sqr\m{A}_\beta$ induced by $\sigma(\gamma, \beta)$.
\end{defn}

The rotating isomorphism enables us to define an appropriate notion of commutative Green functors in the $RO(G)$-graded case.

\begin{defn} An $RO(G)$-graded Green functor $\m A_\star$ with multiplication map $\mu: \m A_\star\sqr \m A_\scalestar \to \m A_\scalestar$  is \emph{commutative} if $\mu \tau= \mu$, where $\tau: \m{A}_\scalestar\sqr\m{A}_\scalestar\to \m{A}_\scalestar\sqr\m{A}_\scalestar$ is the rotating isomorphism. 
\end{defn}

Note that this definition of commutativity is compatible with that of commutative equivariant ring spectra, in the following sense.  

\begin{prop}\label{prop:commGreen}\cite[Theorem 5.1]{LeMa} If $R$ is a commutative ring $G$-spectrum, then $\m \pi_\scalestar^G (R)$ is a commutative $RO(G)$-graded Green functor. 
\end{prop}

\begin{defn}
	Let $G\subset S^1$ be a finite subgroup, let $g\in G,$ and let $\m{R}_\scalestar$ be an $RO(G)$-graded Green functor for $G$ and  $\m{M}_\scalestar$ an $\m{R}_\scalestar$-bimodule. The \emph { twisted cyclic nerve} of $\m{R}_\scalestar$ with coefficients in ${}^g\m{M}_\scalestar$ is the simplicial $RO(G)$-graded Mackey functor with $q$-simplices given by 
	$$[q]\mapsto \m{B}^{cy, G}_q(\m{R}_\scalestar;{}^g\m{M}_\scalestar)={}^g\m{M}_\scalestar\square\m{R}_\scalestar^{\square q}.$$
	For $1\leq i\leq q-1,$ the face map $d_i$ multiplies the $i$th and $(i + 1)$st box factors. The $0$th face map $d_0$ applies the right action of $\m{R}_{\scalestar}$ on $\m{M}_{\scalestar}$. The $q$th face map $d_q$ is given by the composite
	$${}^g\m{M}_\scalestar\square\m{R}_\scalestar^{\square q}\xrightarrow{\tau_q} 
	\m{R}_\scalestar\sqr{}^g\m{M}_\scalestar\square\m{R}_\scalestar^{\square (q-1)}\xrightarrow{{}^g\lambda\sqr\Id}
	{}^g\m{M}_\scalestar\square\m{R}_\scalestar^{\square (q-1)}$$
where ${}^g\lambda$ denotes the $G$-twisted left action of $\m R$ on $\m M$ (Definition \ref{twistaction}). The map $\tau_q$ is given by iterating the rotating isomorphism, moving the last factor to the front. For $0 \leq i \leq  q-1$, the degeneracy map $s_i$ is induced by the unit in the $(i+1)$st factor.
\end{defn}

As in the ungraded case, it is straightforward to verify that, with the sign introduced in the rotating map, the definition above indeed specifies a simplicial object. 

\begin{defn}
	Let $G\subset S^1$ be a finite subgroup, let $g\in G$, and let $\m{R}_\scalestar$ be an $RO(G)$-graded Green functor for $G$ and $\m{M}_\scalestar$ an $\m{R}_\scalestar$-bimodule. The \emph{twisted Hochschild homology of $\m R_\scalestar$ with coefficients in ${}^g\m{M}_\scalestar$}, $\m{\HH}^G_*(\m{R}_\scalestar; {}^g\m{M}_\scalestar)$, is the homology of the twisted cyclic nerve $\m{B}^{cy, G}_\bullet(\m{R}_\scalestar;{}^g\m{M}_\scalestar)$. When $\m M_\scalestar =\m R_\scalestar$ and $g$ is the generator $e^{2\pi i/|G|}$ of $G$, we write $\m{\HH}^G_*(\m{R}_\scalestar)=\m{\HH}^G_*(\m{R}_\scalestar; {}^g\m{R}_\scalestar)$.  
\end{defn}

\begin{rem}
	Let $\m{R}_\scalestar$ be an $RO(G)$-graded Green functor for $G$. When $\m{R}_\scalestar$ is concentrated in degree $0$, the twisted Hochschild homology of the $RO(G)$-graded Green functor $\m{R}_{\scalestar}$ is an $RO(G)$-graded Mackey functor concentrated in degree 0. The degree 0 part coincides with the twisted Hochschild homology of the ungraded Green functor $\m{R}_0$: 
	$$\m{\HH}^G_i(\m{R}_\scalestar)_0\cong \m{\HH}^G_i(\m{R}_0).$$
\end{rem}

Let $\m{E}_\scalestar$ be an $RO(G)$-graded commutative Green functor, and consider the symmetric monoidal category of $\m{E}_\scalestar$-algebras $(\Algcat{\m{E}_\scalestar},\sqr_{{\m{E}}_\scalestar}, \m{E}_\scalestar)$.  As in Section \ref{Hochschild}, we can define the Hochschild homology $\m{\HH}^{\m{E}_\scalestar,G}_i(\m{R}_\scalestar)$ of a $\m{E}_\scalestar$-algebra $\m{R}_\scalestar$ in terms of a relative graded box product.

\subsection{The equivariant B\"okstedt spectral sequence}\label{ss:specsequ}
Before constructing the equivariant B\"ok\-stedt spectral sequence, we briefly  review the classical case. The classical spectral sequence was originally constructed by B\"okstedt, who used it to compute the topological Hochschild homology of $\mathbb{F}_p$ and $\Z$ \cite{bokstedt}. A construction of the B\"okstedt spectral sequence for generalized homology theories in the context of EKMM spectra ($S$-modules) can be found in   \cite{ekmm}.

\begin{thm} \cite[IX, Theorem 2.9]{ekmm} 
	Let $E$ be a commutative ring spectrum, $R$ a ring spectrum, and $M$ a cellular $(R,R)$-bimodule. If $E_*(R)$ is $E_*$-flat, then there is a spectral sequence of the form 
	$$E^2_{p,q}=\HH^{E_*}_{p,q}(E_*(R);E_*(M))\Rightarrow E_{p+q}(\THH(R;M)).$$
\end{thm}

This spectral sequence is a special case of the following general theorem (\cite[X, Theorem 2.9]{ekmm}), which provides a spectral sequence for computing the homology of the geometric realization $\abs{K_\bullet}$ of a simplicial object $K_\bullet$ in spectra.   The construction is based on the simplicial filtration of $\abs{K_\bullet}.$ Recall that a simplicial spectrum $K_\bullet$ is proper if the map from the degenerate subspectrum, $sK_q \to K_q$, is a cofibration for each $q$.

\begin{thm}
\label{ekmmsimplicial}
	Let $K_\bullet$ be a proper simplicial spectrum and $E$ any spectrum. There is a natural homological spectral sequence  
	$$E^2_{p,q}=H_p(E_q(K_\bullet))\Rightarrow E_{p+q}(\abs{K_\bullet}).$$
\end{thm}

In this section we develop an equivariant version of the B\"okstedt spectral sequence to study the equivariant homology of twisted topological Hochschild homology. One can define both integer and $RO(G)$-graded versions of this equivariant B\"okstedt spectral sequence. However, the flatness conditions in the spectral sequence are more likely to hold in the $RO(G)$-graded case, so we focus our discussion on the $RO(G)$-graded spectral sequence. 

Let $E$ be a commutative $G$-spectrum and $R$ a $G$-spectrum. In order to construct the equivariant B\"okstedt spectral sequence converging to the $RO(G)$-graded homology $\m{E}_{\scalestar}(i_G^*\THH_G(R))$, we first need to understand how the $g$-twisting on $R$ interacts with equivariant  $E$-homology.

For a group $G$ and an element $g\in G$, we let $c_g$ denote the conjugation automorphism
$$
c_g: G \to G,
$$
where $c_g(h) = g^{-1}hg$. For a $G$-spectrum $X$, left multiplication by $g$ induces an isomorphism of $G$-spectra (see, for instance, \cite[Section 3.1]{schwede}):
\[
	l_g: c_g^* X \to X: x\mapsto gx.
\]
When $G$ is an abelian group, the conjugation map $c_g:G\to G$ is the identity for every $g\in G$ and the induced endofunctor $c_g^*$ on the category of $G$-spectra is the identity functor, whence $l_g: X \to X$
is  $G$-equivariant for every $G$-spectrum $X$ and every $g\in G$. We often denote this map simply by $g: X\to X$. When $l_g$ is equivariantly homotopic to the identity map, we say that $g$ acts trivially on $X$.

This general observation specializes to the following useful result.

\begin{prop}
\label{prop:action_on_Mackey}
	Let $G$ be an abelian group. For every $G$-spectrum $X$, the $G$-action on $X$ induces a levelwise  $G$-action on the graded Mackey functor $\m{\pi}_{\scalestar}(X).$
\end{prop}

\begin{proof}
 The induced $G$-action on $RO(G)$-graded homotopy Mackey functors $\m{\pi}_\scalestar$ acts grading-wise via the quotient $G/H$, i.e., for each subgroup $H\subset G$ and each $\alpha$ in $RO(G)$, the $G$-action is specified by
$$G\times \m{\pi}_\alpha(X)(G/H)\to G/H\times \m{\pi}_\alpha(X)(G/H)\to \m{\pi}_\alpha(X)(G/H),$$
where the second map is the Weyl group action.    
\end{proof}

\begin{defn} Let $R$ be a $G$-ring spectrum and $g$ an element of $G$. The \emph{$g$-twisted $R$-module structure} on $R$, denoted ${}^gR$, has right action map given by the usual multiplication, $\mu$, and a twisted left action map, ${}^{g}\mu$, given by
$$
\xymatrix{R \wedge R \ar[d]_{g \wedge 1} \ar[dr]^{{}^{g}\mu} & \\
R \wedge R \ar[r]^{\mu} & R.}
$$
\end{defn}

For $E$ a commutative $G$-ring spectrum, and $R$ a $G$-ring spectrum, the $\m{E}_{\scalestar}(R)$-module $\m{E}_{\scalestar}({}^gR)$ will arise in the construction of the equivariant B\"okstedt spectral sequence. In the next lemma, we compare $\m{E}_{\scalestar}({}^gR)$, as an $\m{E}_{\scalestar}(R)$-module, to the twisted module ${}^g\m{E}_{\scalestar}(R)$, where the twisting is defined as in Definition \ref{twistaction}. 

\begin{lem}
\label{TrivialActionLemma} Let $G$ be an abelian group, and let $g\in G$.
 Let $E$ be a commutative $G$-ring spectrum such that $g$ acts trivially on $E$. For every $G$-ring spectrum $R$, the identity is a morphism of left $\m{E}_\scalestar(R)$-modules
 $$^g\m{E}_\scalestar(R) \cong \m{E}_\scalestar(^gR).$$
\end{lem}

\begin{proof} 
It suffices to check $E \sm {}^gR \simeq {}^g(E \sm R)$ as left $ E \sm R$-modules. 

By the assumption, the following diagram homotopy commutes:
\[
\begin{tikzcd}
(E\wedge R) \wedge (E \wedge R) \ar[r,"\Id"]\ar[d,"\mu\circ (g\wedge g\wedge \Id\wedge\Id)"] & (E\wedge R) \wedge (E \wedge R)\ar[d,"\mu\circ (\Id\wedge g\wedge \Id\wedge\Id)"]\\
E \wedge R\ar[r, "\Id"] & E \wedge R
\end{tikzcd}
\]
where the left vertical arrow is the module action of ${}^g(E \sm R)$ and the right one is that of $E \sm {}^gR$. The result follows after passing to homotopy Mackey functors.

\end{proof}

The construction of twisted topological Hochschild homology recalled in Section \ref{sec:relTHH} produces an $S^1$-spectrum. We restrict it to a $G$-spectrum and compute its associated $RO(G)$-graded homology. Viewed as a $G$-spectrum, the twisted topological Hochschild homology of a $G$-ring spectrum is isomorphic to the $G$-equivariant topological Hochschild homology over the twisted module ${}^gR$.

\begin{prop}
	For every $G$-ring spectrum $R$, there is a natural isomorphism of $G$-spectra:
	$$\THH(R;{}^gR) \cong i_G^*\THH_G(R).$$
\end{prop}

\begin{proof}
	 The identity map induces an isomorphism of simplicial $G$-spectra between $B^{cy,G}(R)_\bullet$ and $B^{cy}(R;{}^gR)_\bullet$, the usual cyclic nerve of $R$ with coefficients in the $R$-biomodule ${}^gR$. Passing to the geometric realization, we obtain a $G$-spectrum isomorphism.
\end{proof}

Let $E$ be a commutative $G$-ring spectrum. The simplicial filtration of $B^{cy}(R;{}^gR)_\bullet$
$$\cdots F_{s-1}\subset F_s\subset F_{s+1}\subset \cdots\subset B^{cy}(R;{}^gR)_\bullet,$$ 
gives rise to a spectral sequence
 $$E^1_{s,\scalestar}=\m{E}_{s+\scalestar}(F_s/F_{s-1})\Rightarrow \m{E}_{s+\scalestar}(\THH(R;{}^gR)),$$
which is strongly convergent by \cite[X.2.9]{ekmm}.

Let $\alpha\in RO(G)$. The terms $E^1_{*,\alpha}$ form a chain complex with
$$\m{E}_{s+\alpha}(F_s/F_{s-1})\cong \m E_{s+\alpha}\Big(\Sigma^s\big(B^{cy}(R;{}^gR)_s/\sigma B^{cy}(R;{}^gR)_s\big)\Big), $$
in degree $s$, where $\sigma B^{cy}(R;{}^gR)_\bullet$ denotes the degeneracy subspectrum. 

It follows that  $E^1_{*,\alpha}$ is isomorphic to the normalized chain complex of $\m{E}_\alpha(B^{cy}(R;{}^gR)_\bullet).$

\begin{thm}\label{thm:ss}
	Let $G \subset S^1$ be a finite subgroup and $g= e^{2\pi i/|G|}$ a generator of $G$. Let $R$ be a $G$-ring spectrum and $E$ a commutative $G$-ring spectrum such that $g$ acts trivially on $E$. If $\m{E}_\scalestar(R)$ is flat over $\m{E}_\scalestar$, then there is an equivariant B\"okstedt spectral sequence of the following form:
	$$E_{s,\scalestar}^2=\m{\HH}_{s}^{\m{E}_\scalestar,G}(\m{E}_\scalestar(R))\Rightarrow \m{E}_{s+\scalestar}(i_G^*\THH_G(R)).$$
\end{thm}

\begin{proof}
	 Since $E^1_{*,\alpha}$ is isomorphic to the normalized chain complex of $\m{E}_\alpha(B^{cy}(R;{}^gR)_\bullet)$, it suffices to compute the homology of this chain complex.
	
	Since $\m{E}_\scalestar(R)$ is flat over $\m{E}_\scalestar$, in degree $s$ the chain complex $\m{E}_\scalestar(B^{cy}(R;{}^gR)_\bullet)$ is isomorphic to
	$${B^{cy}_{\m{E}_\scalestar}}(\m{E}_\scalestar (R);\m{E}_\scalestar({}^gR)\big)_s,$$the $\m E_\scalestar$-module of $s$-simplices of the cyclic bar construction in the catgory of $\m E_\scalestar$-modules of $\m{E}_\scalestar (R)$ with coefficients in $\m{E}_\scalestar({}^gR)$.
By Lemma \ref{TrivialActionLemma}, it is also isomorphic as an $\m{E}_\scalestar$-module to 
	$${B^{cy}_{\m{E}_\scalestar}}(\m{E}_\scalestar (R);{}^g\m{E}_\scalestar(R)\big)_s.$$
	Formal diagram chasing shows that the $d_1$ differential of the spectral sequence under this isomorphism can be identified with the differential of the complex computing $\m{\HH}^{{\m{E}_\scalestar,G}}_*(\m{E}_\scalestar R).$ Therefore, we can identify the $E_2$-term of the  spectral sequence above with the  Hochschild homology $\m{\HH}_*^{\m{E}_\scalestar,G}(\m{E}_\scalestar(R)).$ This completes the proof.
\end{proof}

As in the classical case, when the input ring spectrum is actually commutative, the resulting spectral sequence inherits a multiplicative structure.

\begin{prop}\label{prop:algstructure} Let $g\in G$, $E$, and $R$ be as in Theorem \ref{thm:ss}.
 If $R$ is endowed with the structure of a commutative $G$-ring spectrum, then the equivariant B\"okstedt spectral sequence inherits the structure of a spectral sequence of $RO(G)$-graded algebras over $\m E_\scalestar$.
\end{prop}

\begin{proof} It is easy to check that if $R$ is a commutative $G$-ring spectrum, then the cyclic nerve $B^{cy}(R;{}^gR)_\bullet$ is a simplicial object in commutative ring $G$-spectra,  where the levelwise multiplication is defined as usual for a tensor product of algebras. Commutativity of $R$ is not required for the degeneracies to be algebra maps, but is necessary for the face maps to be algebra maps.  The particular case of $d_q$ in simplicial level $q$ relies on the fact that $g$ acts on $R$ as an algebra map. 

Since $E$ is also a commutative ring $G$-spectrum, $E\sm B^{cy}(R;{}^gR)_\bullet$ is a simplicial object in commutative ring $G$-spectra.  Proposition \ref{prop:commGreen} implies that $\m E_\scalestar \big (B^{cy}(R;{}^gR)_\bullet)$ is then a simplicial object in $RO(G)$-graded commutative Green functors over $\m E_\scalestar$.  Normalizing, we obtain a differential graded object in $RO(G)$-graded commutative Green functors over $\m E_\scalestar$; this is the $E^1$-page of the spectral sequence.  Since the simplicial filtration of $B^{cy}(R;{}^gR)_\bullet$ respects its algebra structure, all the differentials in the spectral sequence respect the multiplicative structure as well.
\end{proof}

\subsection{Twisted THH of the real bordism spectrum}\label{ss:MUR}
 In this section we consider $MU_{\mathbb{R}}$,  the $C_2$-equivariant real bordism spectrum of Landweber \cite{Landweber} and Fujii  \cite{Fujii}, for which it is natural to compute $C_2$-relative topological Hochschild homology. We compute the equivariant homology of THH$_{C_2}(MU_{\mathbb{R}})$ using the equivariant B\"okstedt spectral sequence constructed in Section \ref{sec:bokstedt}, converging to 
 $$
  \m{H}^{C_2}_{\scalestar}(i^*_{C_2} \THH_{C_2}(MU_{\mathbb{R}}); \m{\mathbb{F}}_2).
 $$

 It follows from the calculations of $\m{\pi}_{\scalestar}(H\m{\mathbb{F}}_2)$ in \cite{HK} that the Weyl action on $\m{\pi}_{\scalestar}(H\m{\mathbb{F}}_2)$ is trivial. In other words, the generator $g$ of $C_2$ induces the identity map after passing to $RO(C_2)$-graded Mackey functors. In particular, it induces the identity map on $RO(C_2)$-graded homotopy groups. Work by Hu--Kriz \cite{HK} also calculates the $C_2$-equivariant Steenrod algebra. The only element inducing the identity map in the $C_2$-equivariant Steenrod algebra is the unit $1$. Therefore $l_g$ is homotopic to the identity map,  allowing us to use the spectral sequence of Theorem \ref{thm:ss} to make this computation.

 As discussed in Section \ref{sec:bokstedt}, in order to identify the $E^2$-term of the equivariant B\"okstedt spectral sequence with Hochschild homology for Green functors, we need to verify a flatness condition, which turns out to hold in the $RO(C_2)$-graded case.  In particular we prove that $\m{\pi}_{\scalestar}^{C_2}(MU_{\mathbb{R}} \wedge H\m{\mathbb{F}}_2)$ is flat (indeed, free) as a $\m{\pi}_{\scalestar}(H\m{\mathbb{F}}_{2})$-module. For ease of notation, we let $H\m{\mathbb{F}}_{2\scalestar}$ denote $\m{\pi}_{\scalestar}(H\m{\mathbb{F}}_{2})$.
 
 \begin{prop}\label{prop:MUR} The $RO(C_2)$-graded Mackey functor $\m{\pi}_{\scalestar}^{C_2}(MU_{\mathbb{R}} \wedge H\m{\mathbb{F}}_2)$ is free as an $H\m{\mathbb{F}}_{2\scalestar}$-module. In particular
 $$
 \m{\pi}_{\scalestar}^{C_2}(MU_{\mathbb{R}} \wedge H\m{\mathbb{F}}_2)  \cong H\m{\mathbb{F}}_{2\scalestar}[b_1, b_2, \ldots].
 $$ 
Here the degree of $b_i$ is $i\rho$, where $\rho$ denotes the regular representation of $C_2$. 
 \begin{proof}
 Let $V$ be a representation of $C_2$. As in \cite{HHR}, we use the notation 
 $$
 S^0[S^V] = \bigvee_{j \geq 0} (S^V)^{\wedge j}
 $$
 for the free associative algebra on the $V$-sphere. As shown in Corollary 5.18 of Hill-Hopkins-Ravenel \cite{HHR}, building off of work of Araki \cite{Ar79}, there is a weak equivalence
 $$
 MU_{\mathbb{R}} \wedge H\m{\mathbb{F}}_2 \simeq H\m{\mathbb{F}}_2 \wedge \bigwedge_{i\geq 1} S^0[S^{i\rho}].
 $$
 This gives an isomorphism of $RO(C_2)$-graded Green functors
 $$
 \m{\pi}_{\scalestar}^{C_2}(MU_{\mathbb{R}} \wedge H\m{\mathbb{F}}_2)  \cong H\m{\mathbb{F}}_{2\scalestar}[b_1, b_2, \ldots],
 $$
 where the degree of $b_i$ is $i\rho$.
 \end{proof}
  \end{prop}
  
As the appropriate flatness condition holds, the equivariant B\"okstedt spectral sequence for $MU_{\mathbb{R}}$ has the form
$$
E^2_{s,\scalestar} = \m{\HH}^{H\m{\mathbb{F}}_{2\scalestar}, C_2}_s(\m{H}^{C_2}_{\scalestar}(MU_{\mathbb{R}}; \m{\mathbb{F}}_2)) \Rightarrow \m{H}^{C_2}_{s+\scalestar}(i^*_{C_2} \THH_{C_2}(MU_{\mathbb{R}}); \m{\mathbb{F}}_2),
$$
where we are considering $\m{H}^{C_2}_{\scalestar}(MU_{\mathbb{R}}; \m{\mathbb{F}}_2)$ as an $H\m{\mathbb{F}}_{2\scalestar}$-algebra. From Proposition \ref{prop:MUR}, it follows that the $E^2$-term is 
$$
E^2_{*,\scalestar} = \m{\HH}^{H\m{\mathbb{F}}_{2\scalestar}, C_2}_*(H\m{\mathbb{F}}_{2\scalestar}[b_1, b_{2}, \ldots]).
$$

The next proposition is the key to computing this $E^2$-term.

\begin{prop}\label{tor}
Let $G \subset S^1$ be a finite subgroup, and let $g\in G$ be the
generator $e^{2\pi i/|G|}$. Let $\m{R}_{\scalestar}$ be an $RO(G)$-graded commutative Green functor for $G$, and $\m{M}_{\scalestar}$ an associative $\m{R}_{\scalestar}$-algebra. If $\m{M}_{\scalestar}$ is flat as an $\m{R}_{\scalestar}$-module, there is a natural
isomorphism
\[
\m{\HH}_*^{\m{R}_{\scalestar},G}(\m{M}_{\scalestar})\cong \m{\textup{Tor}}_{*, \scalestar}^{\m{M}_{\scalestar}\Box_{\m{R}_{\scalestar}}
  \m{M}_{\scalestar}^{op}}(\m{M}_{\scalestar},{}^{g}\m{M}_{\scalestar}),
\]
where $\m{\textup{Tor}}_{i, \scalestar}$ is the $i$th derived functor of the box product over $\m{R}_{\scalestar}$.  
\end{prop}

  \begin{proof}
  Note that $\m{\HH}_*^{\m{R}_{\scalestar},G}(\m{M}_{\scalestar}) = \m{\HH}_*^{\m{R}_{\scalestar}}(\m{M}_{\scalestar}, \gmM_{\scalestar})$, where the latter is the homology of the ordinary Hochschild complex with bimodule coefficients. By \cite{LeMa} and \cite{Lewis}, since $G$ is a finite group,  the homological behavior of graded $G$-Mackey functors is standard, so that the usual homological algebra argument implies that 
  $$
  \m{\HH}_*^{\m{R}_{\scalestar}}(\m{M}_{\scalestar}, \gmM_{\scalestar}) \cong \m{\textup{Tor}}_{*, \scalestar}^{\m{M}_{\scalestar}\Box_{\m{R}_{\scalestar}}
  \m{M}_{\scalestar}^{op}}(\m{M}_{\scalestar},{}^{g}\m{M}_{\scalestar}).
  $$
 \end{proof}

We now apply the proposition to computing $\m{\HH}^{H\m{\mathbb{F}}_{2\scalestar},C_2}_*(\m{H}^{C_2}_{\scalestar}(MU_{\mathbb{R}}; \m{\mathbb{F}}_2))$. 

\begin{prop}\label{prop:E2}
  The $C_2$-twisted Hochschild homology of the $H\m{\mathbb{F}}_{2\scalestar}$-algebra $\m{H}^{C_2}_{\scalestar}(MU_{\mathbb{R}}; \m{\mathbb{F}}_2)$ is
  $$
 \m{\HH}^{H\m{\mathbb{F}}_{2\scalestar},C_2}_*( \m{H}^{C_2}_{\scalestar}(MU_{\mathbb{R}}; \m{\mathbb{F}}_2)) \cong H\m{\mathbb{F}}_{2\scalestar}[b_1, b_2, \ldots] \hspace{.1cm} \Box_{H\m{\mathbb{F}}_{2\scalestar}} \hspace{.1cm} \Lambda_{H\m{\mathbb{F}}_{2\scalestar}}(z_1, z_2, \ldots),
  $$
  where $|b_i| = (0,i\rho)$ and $|z_i| = (1, i\rho)$.
  \end{prop}
  
  \begin{proof}
  For ease of notation, let $\m{M}_{\scalestar}$ denote $\m{H}^{C_2}_{\scalestar}(MU_{\mathbb{R}}; \m{\mathbb{F}}_2) \cong H\m{\mathbb{F}}_{2\scalestar}[b_1, b_2, \ldots],$ 
$|b_i|= i\rho$, and  let $\m{M}_{\scalestar}^e$ denote $\m{M}_{\scalestar}\Box_{H\m{\mathbb{F}}_{2\scalestar}}
  \m{M}_{\scalestar}^{op}$. By Proposition \ref{tor} above, 
  $$\m{\HH}_*^{H\m{\mathbb{F}}_{2\scalestar},C_2}(\m{M}_{\scalestar}) \cong \m{\textup{Tor}}_{*, \scalestar}^{\m{M}_{\scalestar}^e}(\m{M}_{\scalestar},{}^{g}\m{M}_{\scalestar}),
  $$ 
  where $g$ denotes a generator of $C_2$. 
  
We first consider the $C_2$-action on $\m{M}_{\scalestar}$. As noted earlier, the $C_2$-action on $H\m{\mathbb{F}}_{2\scalestar}$ is trivial. The $C_2$-action on the generator $b_i$ is induced by the action on $S^{i\rho}$, whence  $ g \cdot b_i = b_i$. We conclude that the $C_2$-action on $\m{M}_{\scalestar}$ is trivial, so that 
$$\m{\HH}_*^{H\m{\mathbb{F}}_{2\scalestar},C_2}(\m{M}_{\scalestar}) \cong \m{\textup{Tor}}_{*, \scalestar}^{\m{M}_{\scalestar}^e}(\m{M}_{\scalestar},\m{M}_{\scalestar}).
  $$
  
Using the homological algebra foundations for $RO(C_2)$-graded Mackey functors laid out in \cite{LeMa}, the standard argument of Cartan and Eilenberg \cite[Thm X.6.1]{CaEi} shows that 
  \[
  \m{\textup{Tor}}_{*, \scalestar}^{\m{M}_{\scalestar}^e}(\m{M}_{\scalestar},\m{M}_{\scalestar}) \cong \m{M}_{\scalestar} \Box_{H\m{\mathbb{F}}_{2\scalestar}} \hspace{.1cm} \m{\textup{Tor}}_{*, \scalestar}^{\m{M}_{\scalestar}}(H\m{\mathbb{F}}_{2\scalestar}, H\m{\mathbb{F}}_{2\scalestar}).
  \]
  Using the Koszul complex as in the classical case, one can compute that 
  \[
  \m{\textup{Tor}}_{*, \scalestar}^{H\m{\mathbb{F}}_{2\scalestar}[b_1, b_2, \ldots]}(H\m{\mathbb{F}}_{2\scalestar}, H\m{\mathbb{F}}_{2\scalestar}) \cong \Lambda_{H\m{\mathbb{F}}_{2\scalestar}}(z_1, z_2, \ldots), 
  \]
  where $|z_i| = (1, |b_i|)$. Thus we conclude that 
  \[
  \m{\textup{Tor}}_{*, \scalestar}^{\m{M}_{\scalestar}^e}(\m{M}_{\scalestar},\m{M}_{\scalestar}) \cong H\m{\mathbb{F}}_{2\scalestar}[b_1, b_2, \ldots] \hspace{.1cm} \Box_{H\m{\mathbb{F}}_{2\scalestar}} \hspace{.1cm} \Lambda_{H\m{\mathbb{F}}_{2\scalestar}}(z_1, z_2, \ldots).
  \]
 where $|b_i| = (0,i\rho)$ and $|z_i| = (1, i\rho)$.

  \end{proof}

We now compute the equivariant homology of $\THH_{C_2}(MU_{\mathbb{R}}).$
\begin{thm}\label{thm:MUR}
The $RO(C_2)$-graded equivariant homology of $\THH_{C_2}(MU_{\mathbb{R}})$ is 
$$
\m{H}^{C_2}_{\scalestar}(i^*_{C_2} \THH_{C_2}(MU_{\mathbb{R}}); \m{\mathbb{F}}_2) \cong H\m{\mathbb{F}}_{2\scalestar}[b_1, b_2, \ldots]  \otimes_{\mathbb{F}_{2}}  \Lambda_{\mathbb{F}_{2}}(z_1, z_2, \ldots)
$$
as an $H\m{\mathbb{F}}_{2\scalestar}$-module. Here $|b_i| = i\rho $ and $|z_i| = 1+i\rho.$
\begin{proof}
We use the equivariant B\"okstedt spectral sequence of Theorem \ref{thm:ss}:
$$
E^2_{s,\scalestar} = \m{\HH}^{H\m{\mathbb{F}}_{2\scalestar}, C_2}_s(\m{H}^{C_2}_{\scalestar}(MU_{\mathbb{R}}; \m{\mathbb{F}}_2)) \Rightarrow \m{H}^{C_2}_{s+\scalestar}(i^*_{C_2} \THH_{C_2}(MU_{\mathbb{R}}); \m{\mathbb{F}}_2).
$$
By Proposition \ref{prop:E2}, the $E^2$-term of this spectral sequence is:
$$
E^2_{*, \scalestar} = H\m{\mathbb{F}}_{2\scalestar}[b_1, b_2, \ldots] \hspace{.1cm} \Box_{H\m{\mathbb{F}}_{2\scalestar}} \hspace{.1cm} \Lambda_{H\m{\mathbb{F}}_{2\scalestar}}(z_1, z_2, \ldots),
$$
where $|b_i| = (0, i\rho)$ and $|z_i|= (1,i\rho).$  Note that an element of $RO(C_2)$ has the form $a + b\sigma$, where $\sigma$ denotes the sign representation. Thus we can view the $RO(C_2)$-graded equivariant B\"okstedt spectral sequence for $MU_{\mathbb{R}}$ as a trigraded spectral sequence with integer gradings, and differentials:  
$$
d^r: E^2_{s, a, b} \to E^2_{s-r, a+r-1,b }.
$$
The $C_2$-ring spectrum $MU_{\mathbb{R}}$ is commutative, so by Proposition \ref{prop:algstructure}, this equivariant B\"okstedt spectral sequence is a spectral sequence of $H\m{\mathbb{F}}_{2\scalestar}$-algebras. We observe that all of the $H\m{\mathbb{F}}_{2\scalestar}$-algebra  generators of the $E^2$-term are in filtration less than or equal to 1. Therefore the differentials on the generators are all zero, and hence the spectral sequence collapses. 
To complete the proof we observe that as $H\m{\mathbb{F}}_{2\scalestar}$-modules, 
\[
H\m{\mathbb{F}}_{2\scalestar}[b_1, b_2, \ldots] \hspace{.1cm} \Box_{H\m{\mathbb{F}}_{2\scalestar}} \hspace{.1cm} \Lambda_{H\m{\mathbb{F}}_{2\scalestar}}(z_1, z_2, \ldots) \cong H\m{\mathbb{F}}_{2\scalestar}[b_1, b_2, \ldots]  \otimes_{\mathbb{F}_{2}}  \Lambda_{\mathbb{F}_{2}}(z_1, z_2, \ldots).
\]
\end{proof}

\end{thm}
Since $H\m{\mathbb{F}}_{2\scalestar}$ was computed in Proposition 6.2 of \cite{HK}, Theorem \ref{thm:MUR} provides an explicit description of the homology of THH$_{C_2}(MU_{\mathbb{R}})$.  
  
\section{Twisted topological Hochschild homology of Thom spectra}\label{sec:thomspectra}

Given any map $f: X \to BO$, one can construct its Thom spectrum, whose $n^{th}$ space is the Thom space of  the restriction of $f$ to  $f^{-1}(BO(n))$. More generally, one can define the Thom spectrum $Th(f)$ of a map $f: X \to Pic(R)$, where $R$ is a commutative ring spectrum. A point-set model for $Th(f)$ was described in Chapter IX of \cite{LMS}, and an infinity-categorical model in \cite{ABGHR}. For example, if $f$ is nullhomotopic, $Th(f) \simeq R \wedge X_+$. If $f$ is an $E_n$-map between $E_n$-spaces, then $Th(f)$ is an $E_n$-ring spectrum.

Mahowald showed that $H\mathbb{F}_2$ is the Thom spectrum of a 2-fold loop map $\Omega^2 S^3 \to BO$; a similar description holds for the other Eilenberg--MacLane spectra $H\mathbb{F}_p$, and for $H\mathbb{Z}_{(p)}$. In \cite{BCS},  Blumberg, Cohen, and Schlichtkrull studied the symmetric monoidal properties of the Thom spectrum functor, which they applied to give a simple description of the topological Hochschild homology of these Eilenberg--MacLane spectra.

In this section, we exploit the $G$-symmetric monoidal properties of the equivariant Thom spectrum functor \cite{HKZ}, along with the description of $H \underline{\mathbb{F}}_2$ and $H\underline{ \mathbb{Z}}_{(2)}$ as equivariant Thom spectra in \cite{BW} and \cite{HW}, to compute $\THH_{C_2}$ of these $C_2$-ring spectra.

The first step is to give a ``conjugation action" description, as in \cite{Klang}, of the topological Hochschild homology of equivariant Thom spectra. The idea for this conjugation action comes from the description of the Hochschild homology of a group ring $k[\Gamma]$ as group homology, $H_*(\Gamma; k[\Gamma]^{ad})$, where $k[\Gamma]^{ad}$ indicates that $\Gamma$ acts by conjugation. This description follows from the isomorphism
$$\HH_*(k[\Gamma]) \cong Tor_* ^{k[\Gamma] \otimes k[\Gamma]^{op}} (k[\Gamma], k[\Gamma])$$
by performing a change of rings along the map $k[\Gamma] \to k[\Gamma] \otimes k[\Gamma]^{op}$ sending $\gamma \in \Gamma$ to $\gamma \otimes \gamma^{-1}$. Similarly, one can describe $\THH(\Sigma^\infty _+ \Omega X)$ as the left-derived smash product $B((\Sigma^\infty _+ \Omega X)^{ad}, \Sigma^\infty _+ \Omega X, \mathbb{S})$ by performing a change of rings along the map 
$$\Sigma^\infty _+ \Omega X \to (\Sigma^\infty _+ \Omega X) \wedge (\Sigma^\infty _+ \Omega X)^{op}$$
sending a loop $\gamma \in \Omega X$ to $(\gamma, \gamma^{-1}) \in \Omega X \times (\Omega X)^{op}$. Note that this is the map inducing the conjugation action of $\pi_1(X)$ on itself. In \cite{Klang}, this ``conjugation action" description was generalized to encompass $\THH$ of Thom spectra $Th(f: \Omega X \to BGL_1(R))$, rather than just suspension spectra. Our first goal is to establish a similar description of $\THH$ (and twisted $\THH$) of equivariant Thom spectra.

\subsection{Description of $\THH$.}

\subsubsection*{Construction.} Let $X$ be a pointed G-space, and let $A$ be the Thom spectrum $Th(\Omega f)$ of a loop map $\Omega f: \Omega X \to BO_G$ (or $Pic(R)$ for $R$ a commutative $G$-ring spectrum). We construct an action of $\Sigma^\infty_+ \Omega X$ on $A$ induced by the conjugation action of $\Omega X$ on itself. More explicitly, this action arises from a map of $G$-ring spectra $\Delta_*: \Sigma^\infty _+ \Omega X \to A \wedge A^{op}$, induced on Thom spectra by the following commutative diagram.
$$\xymatrix{
\Omega X \ar[r]^-{\Delta} \ar[rd] & \Omega X \times (\Omega X)^{op} \ar[d]^-{(mult) \circ (\Omega f \times (\Omega f)^{op})} \\
& BO_G
}$$

The map $\Delta$ sends $\gamma \in \Omega X$ to $(\gamma, \gamma^{-1}) \in \Omega X \times (\Omega X)^{op}$. As proven in \cite{HKZ}, $A \wedge A^{op}$ is the Thom spectrum of the vertical map. Its composite with the map $\Delta$ is null, because the concatenation of a loop with its inverse is trivial. Since the Thom spectrum of a null map is the suspension spectrum of the base space, $\Delta$ induces a map of G-ring spectra $\Sigma^\infty _+ \Omega X \to A \wedge A^{op}$. The usual left action of $A \wedge A^{op}$ on $A$ then pulls back to a left action of $\Sigma^\infty _+ \Omega X $ on $A$.

\begin{thm}\label{thm-conj}
Under the action defined above of $\Sigma^\infty_+ \Omega X$ on $A$, there is an equivalence of $G$-spectra
$$\THH(A) \simeq B(A,\Sigma^\infty_+ \Om X, \mathbb{S}).$$
Additionally, if $G = C_n$ with generator $g$,
$$\THH_{C_n}(A) \simeq B(A^g,\Sigma^\infty_+ \Om X, \mathbb{S}).$$
\end{thm}

\begin{proof} We mimic the proof of Theorem 5.7 of \cite{Klang}.
Since $\THH(A) \simeq B(A, A \wedge A^{op}, A)$ and $\THH_{C_n}(A) \simeq B(A^g, A \wedge A^{op}, A)$, by a change of rings, it suffices to show that as an $A$-bimodule,
$$A \simeq B( A \wedge A^{op}, \Sigma^\infty _+ \Om X, \mathbb{S}).$$

The right action of $\Sigma^\infty _+ \Om X$ on $A \wedge A^{op}$ comes from the map of ring spectra, $\Delta_*: \Sigma^\infty _+ \Om X \to A \wedge A^{op}$. Alternatively, $\Sigma^\infty _+ \Om X$ is a Thom spectrum over $\Om X$, and $A \wedge A^{op}$ is a Thom spectrum over $\Om X \times (\Om X)^{op}$; we can describe the action on these spaces. 

Let $\Omega X$ act on $\Om X \times (\Om X)^{op}$ on the right as follows: $\gamma \in \Omega X $ takes $(\alpha, \beta)$ to $(\alpha \gamma,\gamma^{-1}  \beta)$. There is a homotopy fiber sequence
$$\xymatrix{
\Omega X \ar[r]^-\Delta & \Om X \times (\Om X)^{op} \ar[r]^-{mult} & \Om X \ar[r] & X,
}$$
where the last map evaluates at the midpoint of the interval. It follows that  $mult: \Om X \times (\Om X)^{op} \to \Om X$ descends to an equivalence of $\Omega X$ with the homotopy orbits of the $\Om X$-action above on $\Om X \times (\Om X)^{op}$. More explicitly, we can realize this as an equivalence of $G$-spaces endowed with an $\Om X \times (\Om X)^{op}$-action
$$\xymatrix{
B(\Om X \times (\Om X)^{op},\Om X, * ) \ar[r]^-\sim & \Om X,
}$$
which sends a $p$-simplex $((\alpha,\beta), (\gamma_1,...,\gamma_p),*)$ to $\alpha (\gamma_1 ...\gamma_p)(\gamma_p)^{-1} ... (\gamma_1)^{-1} \beta$. (One can use a strictly associative model for $\Om X$, such as Moore loops or the Kan loop group.) The loop space $\Om X \times (\Om X)^{op}$ acts on the left on $B(\Om X \times (\Om X)^{op},\Om X, * )$ by multiplication on $\Om X \times (\Om X)^{op}$ and on the target $\Om X$ by conjugation. 

Since the diagram
$$\xymatrix{
B(\Om X \times (\Om X)^{op},\Om X, * ) \ar[r]^-\sim \ar[d]^-{B(\Om f \times (\Om f)^{op}, *, *)} & \Om X \ar[d]^{\Om f} \\
B(BO_G \times (BO_G)^{op},BO_G, * ) \ar[r]^-\sim & BO_G
}$$
commutes, the Thom spectrum of the top-to-right  composite is equivalent to the Thom spectrum of the left-to-bottom composite. The Thom spectrum functor respects equivalences of $G$-spaces over $BO_G$, respects colimits, and is G-symmetric monoidal (see \cite{HKZ}), so the Thom spectrum of the top-to-right  composite is $A$ and that of the left-to-bottom composite is $B(A \wedge A^{op}, \Sigma^\infty _+ \Om X, \mathbb{S})$.  We obtain
$$B(A \wedge A^{op}, \Sigma^\infty _+ \Om X, \mathbb{S}) \simeq A,$$
as required.
\end{proof}

\subsection{Computation.}

\begin{lem}\label{lem-THH}
As $C_2$-spectra,
$$\THH_{C_2}(H\underline{\mathbb{F}}_2) \simeq \THH(H\underline{\mathbb{F}}_2)$$
and
$$\THH_{C_2}(H\underline{\mathbb{Z}}_{(2)}) \simeq \THH(H\underline{\mathbb{Z}}_{(2)})$$

\end{lem}

\begin{proof}
We will prove this for $H\underline{\mathbb{F}}_2$; the proof for $H\underline{\mathbb{Z}}_{(2)}$ is identical.

 In \cite{dS}, dos Santos constructs a model for $H\underline{\mathbb{F}}_2$ in which the $V^{th}$ space is given by $\mathbb{F}_2[S^V]$. After forgetting to naive $C_2$-spectra, the action of the generator of $C_2$ is trivial, thus the twisting in the twisted cyclic bar construction is trivial, and the result follows.
\end{proof}

In \cite{BW}, Behrens and Wilson showed that $H\underline{\mathbb{F}}_2$ is the Thom spectrum of a $\rho$-fold loop map $\Omega^\rho S^{\rho +1} \to BO_{C_2}$, and Hahn and Wilson  showed in \cite{HW} that $H\underline{\mathbb{Z}}_{(2)}$ is the Thom spectrum of a $(2\sigma + 1)$-fold loop map $\Omega^{2\sigma} (S^{2\sigma + 1}\langle 2 \sigma + 1 \rangle) \to Pic(\mathbb{S}_{(2)})$. Here $\sigma$ denotes the sign representation of $C_2$, $\rho = 1 + \sigma$ denotes the regular representation, and $\mathbb{S}_{(2)}$ denotes the $C_2$-equivariant 2-local sphere spectrum. Hahn and Wilson proved that $S^{2\sigma + 1}$ is the loop space of $\mathbb{H}P^\infty$, whence $S^{2\sigma + 1}\langle 2 \sigma + 1 \rangle \simeq \Omega (\mathbb{H}P^\infty \langle 2\sigma + 2 \rangle)$. Both $H\underline{\mathbb{F}}_2$ and $H\underline{\mathbb{Z}}_{(2)}$  are thus equivariant Thom spectra of loop maps, so we can use Theorem \ref{thm-conj} to compute the topological Hochschild homology of these $C_2$-ring spectra. By Lemma \ref{lem-THH}, this also computes $\THH_{C_2}$.

\begin{thm}\label{thm-THH-thom}
As $C_2$-spectra,
$$\THH(H\underline{\mathbb{F}}_2) \simeq H\underline{\mathbb{F}}_2 \wedge \Omega^\sigma S^{\rho +1}_+$$
and
$$\THH(H\underline{\mathbb{Z}}_{(2)}) \simeq H\underline{\mathbb{Z}}_{(2)} \wedge \Omega^{2\sigma} (\mathbb{H}P^\infty \langle 2\sigma + 2 \rangle) _+$$
\end{thm}

\begin{proof}
We show this for $H\underline{\mathbb{F}}_2$; the proof for $H\underline{\mathbb{Z}}_{(2)}$ is identical. To compute $\THH$, we use Theorem \ref{thm-conj}. As in Section 5 of \cite{Klang}, we show that the $\Sigma^\infty _+\Omega^\rho S^{\rho +1}$-action on $H\underline{\mathbb{F}}_2$ is trivial, i.e., the map $\Sigma^\infty _+\Omega^\rho S^{\rho +1} \to End(H\underline{\mathbb{F}}_2)$ factors as $\Sigma^\infty _+\Omega^\rho S^{\rho +1} \to \mathbb{S} \to End(H\underline{\mathbb{F}}_2)$, where the first map is the augmentation, and the second is adjoint to $id_{H\underline{\mathbb{F}}_2}$. The desired equivalence follows immediately, as $B(\mathbb{S}, \Sigma^\infty _+ \Omega X, \mathbb{S}) \simeq \Sigma^\infty _+ X$. 

The action of $\Sigma^\infty _+\Omega^\rho S^{\rho +1}$ on $H\underline{\mathbb{F}}_2$ is given by the composite
$$\xymatrix{
\Sigma^\infty _+\Omega^\rho S^{\rho +1} \wedge H\underline{\mathbb{F}}_2 \ar[r]^-{\Delta_* \wedge id} & (H\underline{\mathbb{F}}_2 \wedge H\underline{\mathbb{F}}_2 ^{op}) \wedge H\underline{\mathbb{F}}_2 \ar[r]^-{act} & H\underline{\mathbb{F}}_2.
}$$
Since $H\underline{\mathbb{F}}_2$ is $C_2$-commutative, the action
$$\xymatrix{
(H\underline{\mathbb{F}}_2 \wedge H\underline{\mathbb{F}}_2 ^{op}) \wedge H\underline{\mathbb{F}}_2 \ar[r]^-{act} & H\underline{\mathbb{F}}_2
}$$
is equal to
$$\xymatrix{
(H\underline{\mathbb{F}}_2 \wedge H\underline{\mathbb{F}}_2 ^{op}) \wedge H\underline{\mathbb{F}}_2 \ar[r]^-{mult \wedge id} & H\underline{\mathbb{F}}_2 \wedge H\underline{\mathbb{F}}_2 \ar[r]^-{act} & H\underline{\mathbb{F}}_2.
}$$
Since the composite
$$\xymatrix{
\Sigma^\infty _+\Omega^\rho S^{\rho +1} \ar[r]^-{\Delta_*} & (H\underline{\mathbb{F}}_2 \wedge H\underline{\mathbb{F}}_2 ^{op}) \ar[r]^-{mult} & H\underline{\mathbb{F}}_2 
}$$
is induced on Thom spectra by the null map $\Sigma^\infty _+\Omega^\rho S^{\rho +1} \to \Sigma^\infty _+\Omega^\rho S^{\rho +1}$, the action of $\Sigma^\infty _+\Omega^\rho S^{\rho +1}$ on $H\underline{\mathbb{F}}_2$ is trivial, as required.
\end{proof}

Combining this theorem with Lemma \ref{lem-THH}, we complete the desired computations of $C_2$-relative THH.

\begin{cor}\label{cor-twisted-thom}
As $C_2$-spectra,
$$\THH_{C_2}(H\underline{\mathbb{F}}_2) \simeq H\underline{\mathbb{F}}_2 \wedge \Omega^\sigma S^{\rho +1}_+$$
and
$$\THH_{C_2}(H\underline{\mathbb{Z}}_{(2)}) \simeq H\underline{\mathbb{Z}}_{(2)} \wedge \Omega^{2\sigma} (\mathbb{H}P^\infty \langle 2\sigma + 2 \rangle) _+.$$
\end{cor}

Theorem 4.3 of \cite{Hil17} gives an equivariant version of the James splitting, which allows us to explicitly describe $\THH_{C_2}(H\underline{\mathbb{F}}_2)$.

\begin{cor}\label{cor-splitting}
$\THH_{C_2}(H\underline{\mathbb{F}}_2) \simeq H\underline{\mathbb{F}}_2 \wedge (\bigvee_{k \geq 0} S^{2k\rho} \vee \bigvee_{k \geq 0} S^{2k\rho +2})$.
\end{cor}

\bibliographystyle{plain}
\bibliography{bib}
\end{document}